\documentclass[reqno]{amsart}
\usepackage{bbm}
\usepackage{url}
\usepackage{stmaryrd}
\usepackage{mathrsfs}
\usepackage{cases}
\usepackage{amsfonts}
\usepackage{amssymb}
\usepackage{amsmath}
\usepackage{tikz}
\usepackage{extarrows}
\usepackage{enumerate}

\allowdisplaybreaks[4]
\newtheorem{theorem}{Theorem}[section]
\newtheorem{lemma}[theorem]{Lemma}
\newtheorem{proposition}[theorem]{Proposition}
\newtheorem{corollary}[theorem]{Corollary}
\theoremstyle{definition}
\newtheorem{definition}[theorem]{Definition}

\newtheorem{remark}[theorem]{Remark}
\numberwithin{equation}{section}

\let\al=\alpha
\let\b=\beta

\let\d=\delta

\let\la=\lambda
\let\r=\rho
\let\s=\sigma

\let\om=\omega
\let\G= \Gamma

\let\Om=\Omega

\let\th=\theta

\let\ep=\epsilon
\let\va=\varphi
\let\k=\kappa
\def\bbR{\mathbb{R}}
\def\bbS{\mathbb{S}}
\def\bbC{\mathbb{C}}

\newcommand{\be}{\begin{equation*}}
\newcommand{\ee}{\end{equation*}}
\newcommand{\ben}{\begin{equation}}
\newcommand{\een}{\end{equation}}
\newcommand{\bn}{\begin{enumerate}}
\newcommand{\en}{\end{enumerate}}
\newcommand{\bs}{\backslash}

\def\tE{\widetilde{E}}
\def\tF{\widetilde{F}}
\def\tQ{\widetilde{Q}}
\def\tla{\widetilde{\lambda}}

\def\apq{{A_{p,\,q}}}
\def\apn{{A_p}}
\def\bmo{{{\rm BMO}(\mathbb R^n)}}
\def\bmoz{{{\rm BMO}_\lambda(\mathbb R^n)}}
\def\cmo{{{\rm CMO}(\mathbb R^n)}}
\def\fz{{\infty}}
\def\apn{{A_p(\mathbb R^n)}}
\def\rr{{\mathbb R}}
\def\rn{{{\rr}^n}}
\def\ls{\lesssim}
\def\lpwp{{L^p_{w^p}(\rn)}}

\begin{document}
\title[A revisit on the compactness of commutators]
{A revisit on the compactness of commutators}
\author{WEICHAO GUO}
\address{School of Mathematics and Information Sciences, Guangzhou University, Guangzhou, 510006, P.R.China}
\email{weichaoguomath@gmail.com}
\author{HUOXIONG WU}
\address{School of Mathematical Sciences, Xiamen University,
Xiamen, 361005, P.R. China} \email{huoxwu@xmu.edu.cn}
\author{DONGYONG YANG}
\address{School of Mathematical Sciences, Xiamen University,
Xiamen, 361005, P.R. China} \email{dyyang@xmu.edu.cn}

\begin{abstract}
A new characterization of ${\rm CMO}(\mathbb R^n)$ is established by the local mean oscillation.
Some characterizations of iterated compact commutators on weighted Lebesgue spaces are given, which are new
 even in the unweighted setting for the first order commutators.
\end{abstract}
\thanks{Supported by the NSF of China (Nos.11771358, 11471041, 11701112, 11671414), the NSF of Fujian Province of China (Nos.2015J01025, 2017J01011)
and the China postdoctoral Science Foundation (No. 2017M612628).}

\maketitle

\section{Introduction}\label{s1}

Given a locally integrable function $b$, the commutator $[b,T]$
  is defined by
  \be
  [b,T]f(x):=b(x)T(f)(x)-T(bf)(x)
  \ee
 for suitable functions $f$. Let $m\in\mathbb N$.  The iterated commutator $T_b^m$ with $m\geq 2$ is defined by
  \ben\label{e-itera com}
  T_b^m(f):=[b,T_b^{m-1}],\ \ \ T_b^1f:=[b,T]f.
  \een
 The study of equivalent characterization on $L^p(\rn)$-compactness of commutators $[b, T]$
of singular integral operators $T$ was initiated by Uchiyama in his remarkable work \cite{Uchiyama78TohokuMathJ},
where he  improved the notable equivalent characterization of $L^p(\rn)$-boundedness result of Coifman-Rochberg-Weiss \cite{CoifmanRochbergWeiss76AnnMath}.
More precisely,
  let $T_{\Omega}$ be the integral operator with homogeneous kernel $\Omega$ defined by setting, for suitable function $f$
and $x\notin {\rm supp}(f)$,
\begin{equation*}
  T_{\Omega}f(x):=\int_{\mathbb{R}^n}\frac{\Omega(x-y)}{|x-y|^{n}}f(y)dy,
\end{equation*}
where $\Omega$ is a homogeneous function of degree zero, satisfying $\Om\in Lip_1(\mathbb{S}^{n-1})$ and the following mean value zero property:
\begin{equation}\label{mean value zero}
  \int_{\mathbb{S}^{n-1}}\Omega(x')d\sigma(x')=0.
\end{equation}
Coifman, Rochberg and Weiss \cite{CoifmanRochbergWeiss76AnnMath} proved that if a function $b\in {\rm BMO}(\mathbb R^n)$, then the
commutator $[b,T_\Omega]$  is bounded on $L^p(\mathbb R^n)$ for any $p\in(1, \infty)$;
they also proved that, if $[b, R_j]$ is bounded on $L^p(\mathbb R^n)$ for every Riesz transform $R_ j$, $j=1,2,\cdots, n$, then $b\in {\rm BMO}(\mathbb R^n)$.
 In \cite{Uchiyama78TohokuMathJ}, Uchiyama further showed that the commutator $[b,T_{\Omega}]$ is bounded (compact resp.) on $L^p(\rn)$ if and only if
the symbol $b$ is in $\bmo$ ($\cmo$ resp.). Here and in what follows, $\cmo$ is the closure of $C_c^{\infty}(\mathbb{R}^n)$ in the ${\rm BMO}(\mathbb{R}^n)$ topology.
Since then, the work on compactness of commutators of singular and fractional integral operators and its applications to PDE's
have been paid more and more attention; see, for example, \cite{Iwaniec92,KrantzLi01JMAAb,ChenDingWang09PA,Taylor11,ClopCruz13AASFM,ChaffeeTorres15PA,ChenHu15CMB}
and the references therein.


Very recently, Lerner-Ombrosi-Rivera-R\'ios \cite{LernerOmbrosiRivera17arxiv} revisited the boundedness of commutators,
providing a new viewpoint for the Bloom-type characterization of two-weighted boundedness of the iterated bounded commutators on Lebesgue spaces.
In particular, in \cite{LernerOmbrosiRivera17arxiv}, the necessity of boundedness of commutators was proved for a rather wide class of operators,
by a technique deeply depending on the local mean oscillation
instead of mean oscillation of functions.
Inspired by \cite{LernerOmbrosiRivera17arxiv}, in this article, we consider the equivalent compactness of iterated commutators of singular and fractional integral operators
on weighted Lebesgue spaces.
To state our main results, we first recall some notions and notations.

\begin{definition}\label{d-bmo}
The space of functions with bounded mean oscillation, denoted by ${\rm BMO}(\mathbb{R}^n)$, consists of all
$f\in L_{loc}^1(\mathbb{R}^n)$ such that
\be
\|f\|_{\bmo}:=\sup_{Q\subset \mathbb{R}^n}\mathcal{O}(f;Q)<\infty,
\ee
where
\be
f_Q:=\frac{1}{|Q|}\int_{Q}f(y)dy\,\,{\rm and}
\,\,
\mathcal{O}(f;Q):=\frac{1}{|Q|}\int_Q |f(x)-f_Q|dx.
\ee
\end{definition}

\begin{remark}\label{r-bmo norm equiv re-im}
We use $Re(f)$ and $Im(f)$ to denote the real and imaginary part of $f$, respectively.
Observe for a complex-valued $f$, $f\in \bmo$ if and only if $Re(f), Im(f)\in \bmo$.
Moreover, $\|f\|_{\bmo}\sim \|Re(f)\|_{\bmo}+\|Im(f)\|_{\bmo}$.
\end{remark}

The following class of $\apn$ was introduced by Muckenhoupt \cite{Muckenhoupt72TAMS} to
study the the weighted norm inequalities of Hardy-Littlewood maximal operators, and  $\apq$ weights was  introduced by
Muckenhoupt--Wheeden  \cite{MuckenhouptWheeden74TAMS} to
study the weighted norm inequalities of fractional integrals, respectively.

\begin{definition}\label{d-Ap weight}
For $1<p<\infty$, the Muckenhoupt class $A_p$ is the set of locally integrable weights $\om$ such that
\be
[\om]_{A_p}^{1/p}:=\sup_{Q}\left(\frac{1}{|Q|}\int_{Q}\om(x)dx\right)^{1/p}\left(\frac{1}{|Q|}\int_{Q}\om(x)^{1-p'}dx\right)^{1/p'}<\infty,
\ee
where $\frac1p+\frac1{p'}=1$. For $1<p,q<\infty$, $1/q=1/p-\al/n$ with $0<\al<n$,
a weight function $\om$ is called an $A_{p,\,q}$ weight if
$$[\om]_{A_{p,\,q}}^{1/q}:=\sup_Q\left(\frac1{|Q|}\int_Q \om^q(x)dx\right)^{1/q}\left(\frac1{|Q|}\int_Q \om^{-p'}(x)dx\right)^{1/p'}<\fz.$$
\end{definition}

Let $\alpha\in [0,n)$. The singular or fractional integral operator with homogeneous kernel is defined by
\begin{equation}\label{integral operator, linear case}
  T_{\Omega,\,\alpha}f(x):=\int_{\mathbb{R}^n}\frac{\Omega(x-y)}{|x-y|^{n-\alpha}}f(y)dy,
\end{equation}
where $\Omega$ is a homogeneous function of degree zero and satisfies the mean value zero property (\ref{mean value zero}) when $\alpha= 0$.

Our first main result can be formulated as follows.
\begin{theorem}\label{theorem, necessity}
  Let $\om\in \apq$, $1<p,q<\infty$, $0\leq \al<n$, $1/q=1/p-\al/n$, $m\in \mathbb{Z}^+$.
  Let $\Om$ be a bounded measurable function on $\bbS^{n-1}$, which does not change sign and is not equivalent to zero
  on some open subset of  $\bbS^{n-1}$. If $(T_\al)_b^m$ is a compact operator from $L^p(\om^p)$ to $L^q(\om^q)$, then
  $b\in \cmo$.
\end{theorem}

We say that $T_{K_0}$ is a $\r$-type Calder\'{o}n-Zygmund operator on $\bbR^n$ if $T_{K_0}$ is bounded on $L^2$ and it admits the following
representation
\ben\label{representation of singular integral operator}
T_{K_0}f(x)=\int_{\bbR^n}K_0(x,y)f(y)dy\ \ \ \text{for all}\ x\notin \text{supp}f
\een
with kernel $K_0$ satisfying the size condition
\be
|K_0(x,y)|\leq \frac{C_{K_0}}{|x-y|^n}
\ee
and a smoothness condition
\be
|K_0(x,y)-K_0(x',y)|+|K_0(y,x)-K_0(y,x')|\leq \r\left(\frac{|x-x'|}{|x-y|}\right)\frac{1}{|x-y|^n},
\ee
for all $|x-y|>2|x-x'|$, where $\r: [0,1]\rightarrow [0,\infty)$ is a modulus of continuity, that is, $\rho$ is a continuous, increasing,
subadditive function with $\r(0)=0$ and satisfies the following Dini condition:
\be
\int_{0}^1\r(t)\frac{dt}{t}<\infty.
\ee
Similarly, for $\al\in (0,n)$, we define the $\r$-type fractional integral operator by
\ben\label{representation of fractional integral operator}
T_{K_{\al}}f(x):=\int_{\bbR^n}K_{\al}(x,y)f(y)dy\ \ \ \text{for all}\ x\notin \text{supp}f
\een
with kernel $K_{\al}$ satisfying the size condition
\be
|K_{\al}(x,y)|\leq \frac{C_{K_{\al}}}{|x-y|^{n-\al}},
\ee
and the smooth condition
\be
|K_{\al}(x,y)-K_{\al}(x',y)|+|K_{\al}(y,x)-K_{\al}(y,x')|\leq \r\left(\frac{|x-x'|}{|x-y|}\right)\frac{1}{|x-y|^{n-\al}},
\ee
for all $|x-y|>2|x-x'|$, where $\r$ is a modulus of continuity mentioned above.

Note that $T_{K_{\al}}(f)\leq C_{K_{\al}}I_{\al}(|f|)$, where $I_{\al}$ is the classical fractional integral operator defined by
\be
I_{\al}f(x):=\int_{\bbR^n}\frac{f(y)}{|x-y|^{n-\al}}dy.
\ee
Thus, the boundedness of  $T_{K_{\al}}$ is automatically valid by the classical
boundedness of fractional integral operator.

Now we state our second main result.

\begin{theorem}\label{theorem, sufficiency}
  Let $\om\in \apq$, $1<p,q<\infty$, $0\leq \al<n$, $1/q=1/p-\al/n$, $m\in \mathbb{Z}^+$.
  If $b\in \cmo$, then
   $(T_{K_{\al}})_b^m$ is a compact operator from $L^p(\om^p)$ to $L^q(\om^q)$.
\end{theorem}


The rest of the paper is organized as follows. In Section \ref{s2}, we list some basic properties for
weights and local mean oscillations $a_{\la}(b;Q)$ of measurable functions $b$ over cubes $Q$, which are
used in this paper. Section \ref{s3} is devoted to a new characterization of $\cmo$
in terms of the local mean oscillations, which is parallel to the version established by Uchiyama via mean oscillations $\mathcal{O}(f;Q)$
and plays a key role in the proof of Theorem \ref{theorem, necessity}.
The local mean oscillation shows some advantages when one considers lower bound of pointwise regularity of BMO functions,
see also Proposition 3.1 in \cite{LernerOmbrosiRivera17arxiv} or Proposition \ref{proposition, key lower estimates} below.

Section \ref{s4} is devoted to the proof of Theorem \ref{theorem, necessity} and is divided into three subsections.
In Subsection \ref{s4.1}, for any given cube $Q$ and  real-valued measurable function $b$, we construct a function $f$ closely related to $Q$, and
obtain a lower bound of  the weighted $L^q$
norm of $(T_{\Omega,\,\al})_b^m(f)$ over certain subset of $Q$ in terms of $a_{\la}(b;Q)^m$. In Subsection \ref{s4.2},
for $b\in\bmo$ and any cube $Q$, we also obtain an upper bound of  the weighted $L^q$
norm of $(T_{\Omega,\,\al})_b^m(f)$ over the annulus $2^{d+1}Q\bs 2^dQ$ in terms of $2^{-\d dn/p}d^m$, where
$d\in\mathbb N$ large enough, $\delta$ is a positive constant depending on $w\in\apq$ and $f$ is aforementioned.
Using these upper and lower bounds and the new characterization of
$\cmo$ obtained in Section \ref{s3}, we further present the proof of Theorem \ref{theorem, necessity} via a contradiction argument in Subsection \ref{s4.3}.
We remark that our Theorem \ref{theorem, necessity} for $m\geq 2$ is new even
for the $Lip_1$ case in \cite{Uchiyama78TohokuMathJ}.

In Section \ref{s5}, we give the proof of Theorem \ref{theorem, sufficiency}.
For the converse direction, Theorem \ref{theorem, sufficiency} give the new compactness results for iterated commutators.
Thanks to the recent work by Lacey \cite{Lacey17IsraelJMath} and Lerner \cite{Lerner16NewYorkJMath}, our compactness result can be
established in the structure of $\r$-type operators. Note that Theorem \ref{theorem, sufficiency} is new for $m\geq 2$ in all cases.

We would like to point out a class of kernels satisfying both assumption of Theorem \ref{theorem, necessity} and \ref{theorem, sufficiency}.
For the operator $T_{\Omega,\,\al}$, assume that $\Om\in C(\bbS^{n-1})$ is not identically zero. Denote
\be
\r(\d)=\sup_{|x'-y'|\leq \d}|\Om(x')-\Om(y')|.
\ee
If $\r$ satisfy the Dini condition mentioned above,
$T_{\Omega,\,\al}$ satisfies both assumptions in Theorem \ref{theorem, necessity} and \ref{theorem, sufficiency}.

It would be helpful to clarify that in this paper, the kernel $K_{\al}$ and $\Om$ is assumed to be real-valued.
For $m\geq 2$, we only consider the real-valued symbol $b$. This restriction is actually implied in all the previous results of this topic.
Here, we emphasize this to avoid possible misunderstanding.

Finally, we make some conventions on notation. Throughout the paper,  for a real number $a$, $\lfloor a\rfloor$ means the biggest integer no more than $a$.
By $C$ we denote  a {positive constant} which
is independent of the main parameters, but it may vary from line to
line. Constants with subscripts do
not change in different occurrences. Moreover, the symbol $f\lesssim g$ represents that $f\leq Cg$ for some
positive constant $C$. If $f\lesssim g$ and $g\lesssim f$,
we then write $f\sim g$. For a given cube $Q$, we use $c_Q$, $l_Q$ and $\chi_Q$ to denote the center, side length and  characteristic function of $Q$, respectively.

\section{Preliminaries}\label{s2}

\subsection{Some basic properties for non-increasing rearrangements}
Given a measurable function $f$, the non-increasing rearrangement is defined by
\be
f^*(t):=\inf\{\alpha>0: |\{x\in \mathbb{R}^n: |f(x)|>\alpha\}|<t\},\ \ t\in (0,\infty).
\ee
Here, we list some basic properties of non-increasing rearrangement which will be used frequently in the remainder of this paper.
\begin{lemma}\label{l-rearran propt}
\begin{enumerate}[(A)]
\item $f^*(t)$ is left continuous for $t\in (0,\infty)$.
\item $(f+c)*(t)\leq f^*(t)+|c|$,\  $c\in \mathbb{C}$.
\item $(f+g)*(t_1+t_2)\leq f^*(t_1)+g^*(t_2)$,\ $t_1,t_2\in (0,\infty)$.
\item $(f+g)*(t_1+t_2)\leq \max\{f^*(t_1),g^*(t_2)\}$, where $|\text{supp}f\cap \text{supp}g|=0$,  $t_1,t_2\in (0,\infty)$.
\end{enumerate}
\end{lemma}

\begin{proof}
The proofs of (A), (B) and (C) are obvious and we omit the details; see also \cite{Grafakos08}.

To show (D), by the definition of $f^*(t_1)$ and $g^*(t_2)$, for any fixed $\delta>0$, we first have
  \be
  |\{x\in \rn: |f(x)|>f^*(t_1)+\delta\}|<t_1,\ |\{x\in \rn: |g(x)|>g^*(t_2)+\delta\}|<t_2.
  \ee
  From this and the assumption $|\text{supp}f\cap \text{supp}g|=0$,
  \be
  \begin{split}
  &|\{x\in \rn: |f(x)+g(x)|>\max\{f^*(t_1),g^*(t_2)\}+\delta\}|
  \\
 &\quad \leq
  |\{x\in \rn: |f(x)|>f^*(t_1)+\delta\}|+|\{x\in \rn: |g(x)|>g^*(t_2)+\delta\}|<t_1+t_2.
  \end{split}
  \ee
  Thus,
  \be
  (f+g)*(t_1+t_2)\leq \max\{f^*(t_1),g^*(t_2)\}+\delta.
  \ee
  This completes the proof by letting $\delta\rightarrow 0$.
  \end{proof}

We remark that the property (D) in Lemma \ref{l-rearran propt} has its advantage when dealing with a sequence of functions with disjoint supports
(See, for example, the proof of Theorem \ref{theorem, new characterization of CMO} below).

\subsection{Two approaches to $\bmo$}
For any function $f\in L^1_{loc}(\mathbb R^n)$ and cube $Q$, define
\be
\widetilde{\mathcal{O}}(f,Q):=\inf_{c\in \mathbb{C}}\frac{1}{|Q|}\int_Q |f(x)-c|dx.
\ee
It is easy to see that
\be
\widetilde{\mathcal{O}}(f,Q) \leq \mathcal O(f; Q) \leq 2\widetilde{\mathcal O}(f; Q).
\ee

We next consider the characterization of $\bmo$ via local mean oscillations.

\begin{definition}
  By a median value of a real-valued measurable function $f$ over $Q$ we mean a possibly nonunique, real number $m_f(Q)$ such that
\be
|\{x\in Q: f(x)>m_f(Q)\}|\leq |Q|/2
\ee
and
\be
|\{x\in Q: f(x)<m_f(Q)\}|\leq |Q|/2.
\ee
\end{definition}

\begin{definition}
For a complex-valued measurable function $f$, we define the local mean oscillation of $f$ over a cube $Q$ by
\be
\overline{a}_{\lambda}(f;Q):=\inf_{c\in \mathbb{C}}((f-c)\chi_Q)^*(\lambda|Q|)\hspace{6mm}(0<\lambda<1).
\ee
\end{definition}
\begin{definition}
For a real-valued measurable function $f$, we define the local mean oscillation of $f$ over a cube $Q$ by
\be
a_{\lambda}(f;Q):=((f-m_f(Q))\chi_Q)^*(\lambda|Q|)\hspace{6mm}(0<\lambda<1).
\ee
We also define
\be
\widetilde{a}_{\lambda}(f;Q):=\inf_{c\in \mathbb{R}}((f-c)\chi_Q)^*(\lambda|Q|)\hspace{6mm}(0<\lambda<1).
\ee
\end{definition}

The following lemma on the property of the local mean oscillations is useful in Section \ref{s3}; see \cite{John65,Stromberg79IUMJ,Lerner11Lecture}.
We give the proof for completeness.

\begin{lemma}\label{l-local mean osci equiv}
Let $\lambda\in(0, \frac12]$. The following statements are true.
\begin{enumerate}
  \item [{\rm (i)}] For any real-valued function $f$ and cube $Q$,
  \be
\widetilde{a}_{\lambda}(f;Q)\leq a_{\lambda}(f;Q)\leq 2\widetilde{a}_{\lambda}(f;Q).
\ee
  \item [{\rm (ii)}] For any complex-valued function $f$, denote
\be
\|f\|_{\bmoz}:=\sup_Q\overline{a}_{\lambda}(f;Q).
\ee
Then we have
\be
\|f\|_{\bmoz}\sim \|f\|_{\bmo}\sim \sup_Q\widetilde{a}_{\lambda}(Re(f);Q)+\sup_Q\widetilde{a}_{\lambda}(Im(f);Q).
\ee
\end{enumerate}

\end{lemma}

\begin{proof}
To see (i), by the definition of median number, for any real-valued measurable, we have
\be
|\{x\in Q: |f(x)|\geq |m_f(Q)|\}|\geq |Q|/2,
\ee
which implies for suitable $m_f(Q)$,
\be
|m_f(Q)|\leq (f\chi_Q)^*(|Q|/2).
\ee
For a fixed real number $c$, observe that $m_f(Q)-c$ is also a median value of $f-c$.
Replacing $f$ by $f-c$ in the above inequality, we get
\ben
|m_f(Q)-c|\leq ((f-c)\chi_Q)^*(|Q|/2).
\een
Thus, for any $\lambda\in(0, \frac12]$ and $c\in \bbR$,
\be
\begin{split}
((f-m_f(Q))\chi_Q)^*(\lambda|Q|)
\leq &
((f-c)\chi_Q)^*(\lambda|Q|)+|m_f(Q)-c|
\\
\leq &
((f-c)\chi_Q)^*(\lambda|Q|)+((f-c)\chi_Q)^*(|Q|/2)
\\
\leq &
2((f-c)\chi_Q)^*(\lambda|Q|).
\end{split}
\ee
Then (i) follows from the above inequality immediately.

We now show (ii). A direct conclusion by Chebyshev's inequality shows
\be
((f-c)\chi_Q)^*(\lambda|Q|)\leq \frac{1}{\lambda}\frac{1}{|Q|}\int_Q|f-c|dx,\ c\in \mathbb{C},
\ee
which implies that for all complex-valued measurable function $f$ and $c\in \mathbb{C}$,
\be
\overline{a}_{\lambda}(f;Q)\leq\frac{1}{\lambda}\widetilde{\mathcal O}(f;Q).
\ee
Thus,
\ben\label{e-J-S low bdd}
\sup_Q\overline{a}_{\lambda}(f;Q)\leq \frac1\lambda\|f\|_{\bmo}.
\een
The remarkable results of John \cite{John65} and Str\"{o}mberg \cite{Stromberg79IUMJ}
show that the converse direction of the above inequality is still true for $\lambda\in (0,1/2]$,
that is,
\begin{equation}\label{e-J-S upp bdd}
\|f\|_{\bmo}\lesssim \sup_Q\overline{a}_{\lambda}(f;Q).
\end{equation}
Combining \eqref{e-J-S low bdd} and \eqref{e-J-S upp bdd},  we  have the John-Str\"{o}mberg equivalence
\ben\label{e-J-S equiv}
\|f\|_{\bmo}\sim \|f\|_{\bmoz}.
\een
Moreover, observe that $\overline{a}_{\lambda}(Re(f);Q)=\widetilde{a}_{\lambda}(Re(f);Q)$
and $\overline{a}_{\lambda}(Im(f);Q)=\widetilde{a}_{\lambda}(Im(f);Q)$. By Remark \ref{r-bmo norm equiv re-im},
We then conclude that
\begin{align*}
\|f\|_{\bmoz}\sim \|f\|_{\bmo}&\sim \|Re(f)\|_{\bmo}+\|Im(f)\|_{\bmo}\\
&\sim\sup_Q\widetilde{a}_{\lambda}(Re(f);Q)+\sup_Q\widetilde{a}_{\lambda}(Im(f);Q).
\end{align*}
%
%
%
%
%
Thus, the proof of (ii) is completed.
\end{proof}

\subsection{$A_p$ and $A_{p,q}$ weights}
In this subsection, we recall some useful properties of $A_p$ and $\apq$ weights; see \cite{MuckenhouptWheeden74TAMS,Grafakos08,HytonenPerez13AnalPDE,HolmesRahmSpencer16StudiaMath}.
Define the $A_{\infty}$ class of weights by $A_{\infty}:=\cup_{p>1}A_p$, and recall the Fujii-Wilson $A_{\infty}$ constant
\be
[\om]_{A_{\infty}}:= \sup_Q\frac{1}{\om(Q)}\int_Q M(\chi_Q\om)\,dx,
\ee
where $M$ is the Hardy-Littlewood maximal operator.
\begin{lemma}\label{l-Ap weight prop}
Let $p\in(1, \infty)$ and $w\in A_p$.
\begin{enumerate}
  \item [{\rm(i)}] For every $0<\al<1$, there exists $0<\b<1$ such that for every $Q$ and every measurable set $E\subset Q$
with $|E|\geq \al |Q|$, then
\be
\om(E)\geq \b\om(Q).
\ee
  \item [{\rm (ii)}]For all $\la>1$, and all cubes $Q$,
\be
\om(\la Q)\leq \la^{np}[\om]_{A_p}\om(Q).
\ee
  \item [{\rm (iii)}]  $[\om]_{A_{\infty}}\leq c_n[\om]_{A_p}$.

  \item[{\rm (iv)}]There exist a constant $\ep_n$ only depend on $n$ such that
 if $0<\ep\leq \ep_n/[\om]_{A_\infty}$, then $\om$ satisfies the reverse H\"{o}lder  inequality for any cube $Q$,
\be
\left(\frac{1}{|Q|}\int_{Q}\om(x)^{1+\ep}dx\right)^{\frac{1}{1+\ep}}\leq \frac{2}{|Q|}\int_{Q}\om(x)dx.
\ee

\item[{\rm(v)}] There exists a small positive constant $\ep$ depend only on $n$, $p$ and $[\om]_{A_p}$ such that
\be
\om^{1+\ep}\in A_p,\ \ \  \om\in A_{p-\ep}.
\ee

\end{enumerate}
\end{lemma}

\begin{lemma}\label{l-Apq weigh prop}
Let $1<p,q<\infty$, $1/q=1/p-\al/n$ with $0<\al<n$ and $w\in \apq$.
\begin{itemize}
  \item [{\rm(i)}] $\om^p\in\apn$, $\om^q\in A_q$ and  $\om^{-p'}\in A_{p'}$
  \item [{\rm (ii)}]
$
\om\in \apq\Longleftrightarrow \om^q\in A_{q\frac{n-\al}{n}} \Longleftrightarrow \om^q\in A_{1+\frac{q}{p'}}
\Longleftrightarrow \om^{-p'}\in A_{1+\frac{p'}{q}}.
$
\end{itemize}

\end{lemma}

\section{Characterization of $\cmo$ in terms of local mean oscillations}\label{s3}

In this section, we establish a new characterization of $\cmo$ via the local mean oscillation,
which is a key tool in the proof of Theorem \ref{theorem, necessity} and of interest on its own.
  We begin with recalling the classical characterization of $\cmo$ obtained by Uchiyama \cite{Uchiyama78TohokuMathJ} in terms of mean oscillation.

\begin{lemma}[\cite{Uchiyama78TohokuMathJ}]\label{lemma, old characterization of CMO}
  Let $f\in \bmo$. Then $f\in \cmo$ if and only if the following three conditions hold:
  \bn
  \item $\lim_{r\rightarrow 0}\sup\limits_{|Q|=r}{\mathcal{O}}(f;Q)=0$,
  \item $\lim_{r\rightarrow \infty}\sup\limits_{|Q|=r}{\mathcal{O}}(f;Q)=0$,
  \item $\lim_{x\rightarrow \infty}\mathcal O(f;Q+x)=0$.
  \en
\end{lemma}

Before giving our new characterization of $\cmo$, we first give the following result on
the regularity of median value for measurable functions.

\begin{proposition}\label{proposition, approximation by local mean oscillation}
  Let $\lambda\in (0,1/2)$, $f$ be a real-valued measurable function, $Q, \tQ$ be two cubes satisfying $\tQ\subset Q$.
  Suppose that $a_{\lambda}(f;P)<\epsilon$
  for all cubes $P\subset Q$ with $|P|\geq |\tQ|$.
  Then, the following estimate is valid:
  \be
  |m_f(Q)-m_f(\tQ)|\leq \left(1+\left\lfloor\log_{(1/2+\lambda)^{-1}}\frac{|Q|}{|\tQ|}\right\rfloor\right)\epsilon.
  \ee
  Specially, if $|\tQ|>(1/2+\lambda)|Q|$, we have
  \be
  |m_f(Q)-m_f(\tQ)|\leq {a}_{\lambda}(f;Q).
  \ee
\end{proposition}
\begin{proof}
Take $N:=\lfloor\log_{(1/2+\lambda)^{-1}}\frac{|Q|}{|\tQ|}\rfloor$, we can find a sequence of cubes $\{Q_i\}_{i=1}^N$ satisfying
\be
Q\supset Q_1\supset Q_2\supset \cdots \supset Q_N\supset \tQ,\ \  |Q_i|=(1/2+\lambda)^i|Q|\ \ (i=1,2,\cdots,N)
\ee
and
\be
(1/2+\lambda)^N|Q|\geq |\tQ|> (1/2+\lambda)^{N+1}|Q|.
\ee
By the definition of non-increasing rearrangement we have
\be
|\{x\in Q: |f(x)-m_f(Q)|\geq {a}_{\lambda}(f;Q)\}|\leq \la|Q|.
\ee
It implies that
\be
|\{x\in Q: |f(x)-m_f(Q)|\leq {a}_{\lambda}(f;Q)\}|\geq (1-\la)|Q|.
\ee
Then,
\be
\begin{split}
  |\{x\in Q_1: |f(x)-m_f(Q)|\leq {a}_{\lambda}(f;Q)\}|
  \geq &
  (1-\la)|Q|-(|Q|-|Q_1|)
  \\
  = &
  |Q_1|-\lambda|Q|
  \\
  = &
  |Q_1|-\frac{\lambda}{1/2+\lambda}|Q_1|
    \\
  = &
  \frac{1/2}{1/2+\lambda}|Q_1|>\frac{|Q_1|}{2}.
\end{split}
\ee
This and the definition of $m_f(Q_1)$ yield that
\be
|\{x\in Q_1: f(x)\geq m_f(Q)-{a}_{\lambda}(f;Q)\}|> \frac{|Q_1|}{2}\geq|\{x\in Q_1: f(x)> m_f(Q_1)\}|
\ee
and
\be
|\{x\in Q_1: f(x)\leq m_f(Q)+{a}_{\lambda}(f;Q)\}|> \frac{|Q_1|}{2}\geq|\{x\in Q_1: f(x)< m_f(Q_1)\}|.
\ee
Hence,
\be
m_f(Q)-{a}_{\lambda}(f;Q)\leq m_f(Q_1)\leq  m_f(Q)+{a}_{\lambda}(f;Q).
\ee
We have $|m_f(Q)-m_f(Q_1)|\leq {a}_{\lambda}(f;Q)$.
A similar argument yields that
\be
|m_f(Q_i)-m_f(Q_{i+1})|\leq {a}_{\lambda}(f;Q_i),\ (i=1,2,\cdots,N-1)
\ee
and
\be
|m_f(Q_N)-m_f(\tQ)|\leq {a}_{\lambda}(f;Q_N).
\ee
Combining the obtained estimates, we get
\be
\begin{split}
  |m_f(Q)-m_f(\tQ)|
  \leq &
  |m_f(Q)-m_f(Q_1)|+\sum_{i=1}^{N-1}|m_f(Q_i)-m_f(Q_{i+1})|+|m_f(Q_N)-m_f(\tQ)|
  \\
  \leq &
  {a}_{\lambda}(f,Q)+\sum_{i=1}^{N}{a}_{\lambda}(f,Q_i)<(N+1)\epsilon
\end{split}
\ee
which completes this proof.
\end{proof}

Now we state our new characterization of $\cmo$ in terms of local mean oscillation.

\begin{theorem}\label{theorem, new characterization of CMO}
  Let $f\in \bmo$. Then $f\in \cmo$ if and only if the following three conditions hold:
  \bn
  \item $\lim_{r\rightarrow 0}\sup\limits_{|Q|=r}\overline{a}_{\la}(f;Q)=0$,
  \item $\lim_{r\rightarrow \infty}\sup\limits_{|Q|=r}\overline{a}_{\la}(f;Q)=0$,
  \item $\lim_{d\rightarrow \infty}\sup\limits_{Q\cap [-d,d]^n=\emptyset}\overline{a}_{\la}(f;Q)=0$.
  \en
\end{theorem}
\begin{proof}
By the conditions (1),(2) and (3) mentioned in Lemma \ref{lemma, old characterization of CMO}, we can deduce
\be
\lim_{d\rightarrow \infty}\sup\limits_{Q\cap [-d,d]^n=\emptyset}{\mathcal{O}}(f;Q)=0.
\ee
Combining this with Lemma \ref{lemma, old characterization of CMO} and the fact
\be
\overline{a}_{\lambda}(f;Q)\lesssim \frac{1}{\lambda}{\mathcal{O}}(f;Q),
\ee
we complete the proof of "only if" part.

For the "if" part, we follow the arguments in \cite{Uchiyama78TohokuMathJ} with some careful technical modifications fitting our cases.
Conditions (1)-(3) are now assumed to be true.
Since
\be
\widetilde{a}_{\lambda}(Re(f);Q)\leq \overline{a}_{\lambda}(f;Q),\ \
\widetilde{a}_{\lambda}(Im(f);Q)\leq \overline{a}_{\lambda}(f;Q).
\ee
It yields that the real-valued function $Re(f)$ and $Im(f)$ also satisfy the conditions (1)-(3).
Once we verify both $Re(f)$ and $Im(f)$ are $\cmo$ functions, $f$ is also in $\cmo$.
Thus, at the remainder of this proof, we only need to deal with the real-valued $f$.
We use $R_i$ to denote the cube $[-2^{i},2^i]^n$.

For a fixed small number $\epsilon>0$, there exist three constants
depend on $\ep$, denoted by integers $i_{\ep}$, $j_{\ep}$ and $k_{\ep}$ respectively, satisfying $i_{\ep}+2\leq k_{\ep}$,
\be
\sup\{{a}_{\lambda}(f;Q): |Q|\leq 2^{i_{\ep}+1}\}<\ep,\ \sup\{{a}_{\lambda}(f;Q): |Q|\geq 2^{j_{\ep}}\}<\ep,
\ee
and
\be
\sup\{{a}_{\lambda}(f;Q): Q\cap R_{k_{\ep}}=\emptyset\}<\ep.
\ee
Let
\be
d_1:=d_1(\epsilon)=k_{\ep}+1.
\ee
For $x\in R_{d_1}$, $Q_x$ means the dyadic cube of side length $2^{i_{\ep}}$ that contains $x$.
If $x\in R_m\bs R_{m-1}$ for $m>d_1$, $Q_x$ means the dyadic cube with side length $2^{i_{\ep}+m-d_1}$.
Using condition (2) and Proposition \ref{proposition, approximation by local mean oscillation}, we can find a sufficient large $d_2\geq j_{\ep}$
such that
\be
|m_f(Q_x)-m_f(R_{m})|<\frac{\ep}{2},\ x\in R_{m}\bs R_{m-1},\ \ m\geq d_2,
\ee
and
\be
|m_f(R_m)-m_f(R_{m-1})|<\ep,\ \ m\geq d_2.
\ee
Set $g_{\ep}(x):=m_f(Q_x)$ for $x\in R_{d_2}$, the above inequality implies $|g_{\ep}(x)-g_{\ep}(y)|<\ep$ for all $x,y\in R_{m}\bs R_{m-1}$
with $m\geq d_2$.
Take
\be
d_3:=d_2+4+\left\lfloor\log_2\sqrt[n]{\frac{4}{1-2\la}}\right\rfloor.
\ee
For $x\in R_m\bs R_{m-1}$ with $d_2<m\leq d_3$, we set $g_{\ep}(x):=m_f(Q_x)$. And, we set $g_{\ep}(x):=m_f(R_{d_3})$ for $x\in R_{d_3}^c$.

By Proposition \ref{proposition, approximation by local mean oscillation} and
the construction of $g_{\ep}$, if $\overline{Q_x}\cap \overline{Q_y}\neq \emptyset$ or $x,y\in R_{m}\bs R_{m-1}$ with $d_2\leq m\leq d_3$, then
\ben \label{proof, 1}
|g_{\ep}(x)-g_{\ep}(y)|<C\ep
\een
where $C$ is a constant independent of $\epsilon$, but may depends on $\lambda$.
Take a $C_c^{\infty}(\mathbb{R}^n)$ function $\va$ supported on $B(0,1)$, satisfying $\|\va\|_{L^1}=1$
Set $\va_{t}(x):=\frac{1}{t^n}\va(\frac{x}{t})$, $t\in (0,\infty)$.
It follows from (\ref{proof, 1}) that
\be
\begin{split}
|(g_{\ep}\ast \va_t)(x)-g_{\ep}(x)|
= &
|\int_{B(0,1)}(g_{\ep}(x-ty)-g_{\ep}(x))\va(y)dy|
\\
\leq &
\sup_{x\in \rn,y\in B(0,1)}|g_{\ep}(x-ty)-g_{\ep}(x)|\ls\ep
\end{split}
\ee
for sufficient small $t$.
Note $g_{\ep,t}:=g_{\ep}\ast \va_t-m_f(R_{d_3})$ is a $C_c^{\infty}(\mathbb{R}^n)$ function,
we deduce that
\ben\label{proof, 3}
\|g_{\ep,t}-g_{\ep}\|_{\bmo}=
\|(g_{\ep}\ast \va_t-m_f(R_{d_3}))-g_{\ep}\|_{\bmo}\lesssim \|g_{\ep}\ast \va_t-g_{\ep}\|_{L^{\infty}}\lesssim \ep.
\een
Let
\be
\tla:=\frac{2\la+1}{4}.
\ee
We will verify that for all cubes $Q$,
\ben\label{proof, 2}
a_{\tla}(f-g_{\ep};Q)\ls\ep.
\een
This part is divided into following three cases.

\textbf{Case 1.} $Q\subset R_{d_3}$. We further consider the following two cases.\\
When
\be
\max\{l(Q_x): Q_x\cap Q\neq \emptyset\}\geq 2l(Q).
\ee
If $Q\cap R_{d_1}\neq \emptyset$, then $|Q|\leq 2^{ni_\ep}$.
From this and the definition of $i_{\ep}$, we obtain ${a}_{\lambda}(f;Q)<\ep$.\\
If $Q\cap R_{d_1}= \emptyset$, by the definition of $k_{\ep}$ we have ${a}_{\lambda}(f;Q)<\ep$.
Observe that
\be
\sharp\{Q_x: Q_x\cap Q\neq \emptyset\}\ls1.
\ee

Thus, $|g_{\ep}(x)-g_{\ep}(y)|<C\ep$ for any two cubes having nonempty intersection with $Q$.
This implies $|g_{\ep}(x)-(g_{\ep})_Q(x)|<C\ep$ for all $x\in Q$.
So,
\be
\begin{split}
  a_{\tla}(f-g_{\ep};Q)
  \sim
  \tilde a_{\tla}(f-g_{\ep};Q)
  \leq &
  ((f-g_{\ep}+(g_{\ep})_Q-m_f(Q))\chi_Q)^*(\tla|Q|)
  \\
  \leq &
  ((|f-m_f(Q)|+C\ep)\chi_Q)^*(\tla|Q|)
  \\
  \leq &
  ((f-m_f(Q))\chi_Q)^*(\la|Q|)+C\ep
  =
  {a}_{\lambda}(f;Q)+C\ep\ls\ep,
\end{split}
\ee
where in the last inequality we use Property (B) in Lemma \ref{l-rearran propt}.

When
\be
\max\{l(Q_x): Q_x\cap Q\neq \emptyset\}\leq 2 l(Q),
\ee
take
\be
v=v(\lambda):=\max\left\{3+\left\lfloor\log_2\left(\sqrt[n]{\frac{1+2\lambda}{4\lambda}}-1\right)^{-1}\right\rfloor, 0\right\}.
\ee
Let $Q_x^{v}$ be the sub dyadic cube of $Q_x$ with $l(Q_x^{v})=2^{-v}l(Q_x)$.
Observe that if $Q_x^{v}\subset Q$,
\be
l(Q_x^{v})=2^{-v}l(Q_x)\leq 2^{1-v}l(Q),
\ee
this implies that
\be
\begin{split}
  \la\sum_{Q_x^v\cap Q\neq \emptyset}|Q_x^v|
  \leq &
  \la(1+2\cdot 2^{1-v})^n|Q|
  \\
  = &
  \la(1+2^{2-v})^n|Q|
  \leq
  \la\left(1+\sqrt[n]{\frac{1+2\la}{4\la}}-1\right)^n|Q|=\tla|Q|,
\end{split}
\ee
where in the last inequality we use the fact $2^{2-v}\leq \sqrt[n]{\frac{1+2\lambda}{4\lambda}}-1$ implied by the choice of $v$.
From this, we obtain that
\be
\begin{split}
  a_{\tla}(f-g_{\ep},Q)
  = &
  ((f-g_{\ep})\chi_Q)^{*}(\tla |Q|)
  \\
  = &
  (\sum_{Q^v_x\cap Q\neq \emptyset}(f-m_f(Q_x))\chi_{Q_x^v\cap Q})^{*}(\tla |Q|)
  \\
  \leq &
  (\sum_{Q^v_x\cap Q\neq \emptyset}(f-m_f(Q_x))\chi_{Q_x^v})^{*}(\tla |Q|)
    \\
  \leq &
  (\sum_{Q_x^v\cap Q\neq \emptyset}(f-m_f(Q_x))\chi_{Q_x^v})^{*}(\lambda\sum_{Q_x^v\cap Q\neq \emptyset}|Q_x^v|)
  \\
  \leq &
  \max_{Q_x^v\cap Q\neq \emptyset}((f-m_f(Q_x))\chi_{Q_x^v})^{*}(\lambda|Q_x^v|),
\end{split}
\ee
where in the last inequality we use Property (D) in Lemma \ref{l-rearran propt}.
By the definition of $Q_x$, for any cube $\tQ\subset Q_x$,
\be
{a}_{\lambda}(f,\tQ)\ls\ep.
\ee
This and Proposition \ref{proposition, approximation by local mean oscillation} imply that
\be
|m_f(Q_x)-m_f(Q_x^v)|\ls\left(1+\left\lfloor\log_{(1/2+\lambda)^{-1}}\frac{|Q_x|}{|Q_x^v|}\right\rfloor\right)\epsilon\ls\epsilon.
\ee
Thus,
\be
\begin{split}
((f-m_f(Q_x))\chi_{Q_x^v})^{*}(\lambda|Q_x^v|)
\leq &
((f-m_f(Q_x^v))\chi_{Q_x^v})^{*}(\lambda|Q_x^v|)+|m_f(Q_x)-m_f(Q_x^v)|
\\
= &
{a}_{\lambda}(f,Q_x^v)+|m_f(Q_x)-m_f(Q_x^v)|\ls\ep.
\end{split}
\ee
By the above estimates, we get $a_{\tla}(f-g_{\ep},Q)\ls\ep$.

\textbf{Case 2.} $Q\subset R_{d_2}^c$.\\
By the definition of $d_2$ and $d_3$, we obtain $|g_{\ep}(x)-g_{\ep}(x)|<C\ep$ for all $x,y \in R_{d_2}^c$,
which implies that
\be
|(g_{\ep})_Q(x)-g_{\ep}(x)|<C\ep,\ \ x\in R_{d_2}^c.
\ee
From this and the fact ${a}_{\lambda}(f;Q)<C\ep$, we obtain
\be
\begin{split}
 a_{\tla}(f-g_{\ep};Q)
  \sim
  \tilde a_{\tla}(f-g_{\ep};Q)
  \leq &
  ((f-g_{\ep}+(g_{\ep})_Q-m_f(Q))\chi_Q)^*(\tla|Q|)
  \\
  \leq &
  ((f-m_f(Q))\chi_Q)^*(\la|Q|)+C\ep
  =
  {a}_{\lambda}(f;Q)+C\ep\ls\ep,
\end{split}
\ee

\textbf{Case 3.} $Q\cap R_{d_2}\neq \emptyset$, $Q\cap R_{d_3}^c\neq \emptyset$.\\
Firstly, we claim that
\be
\tla|Q|-2|Q\cap R_{d_2}|\geq \la|Q|.
\ee
Since $Q\cap R_{d_2}\neq \emptyset$ and $Q\cap R_{d_3}^c\neq \emptyset$, then $l(Q)\geq 2^{d_3}-2^{d_2}\geq 2^{d_3-1}$.
This and the choice of $d_3$ yields that
\be
\begin{split}
  (\tla-\la)|Q|
  \geq &
  \frac{1-2\la}{4}\cdot 2^{n(d_3-1)}
  \\
  \geq &
  \frac{1-2\la}{4}\cdot \frac{4}{1-2\la}\cdot 2^{n(d_2+2)}\geq 2|R_{d_2}|\geq 2|Q\cap R_{d_2}|.
\end{split}
\ee
By a similar argument in the last case,
\be
|(g_{\ep})_{Q\cap R_{d_2}^c}(x)-g_{\ep}(x)|<C\ep,\ \ x\in R_{d_2}^c.
\ee
Observe that
\be
l(Q)\geq 2^{d_3}-2^{d_2}\geq 2^{d_2}\geq 2^{j_{\ep}}.
\ee
By the definition of $j_{\ep}$, we get
\be
{a}_{\lambda}(f;Q)<\ep.
\ee
Combining this with the above estimates in this case, we obtain
\be
\begin{split}
a_{\tla}(f-g_{\ep};Q)
\sim &
\tilde a_{\tla}(f-g_{\ep};Q)
\\
\leq &
((f-m_f(Q)-g_{\ep}+(g_{\ep})_{Q\cap R_{d_2}^c}(x))\chi_Q)^*(\tla|Q|)
\\
\leq &
((f-m_f(Q)-g_{\ep}+(g_{\ep})_{Q\cap R_{d_2}^c}(x))\chi_{Q\cap R_{d_2}^c})^*(\tla|Q|-2|Q\cap R_{d_2}|)
\\
& +
((f-m_f(Q)-g_{\ep}+(g_{\ep})_{Q\cap R_{d_2}^c}(x))\chi_{Q\cap R_{d_2}})^*(2|Q\cap R_{d_2}|)
\\
= &
((f-m_f(Q)-g_{\ep}+(g_{\ep})_{Q\cap R_{d_2}^c}(x))\chi_{Q\cap R_{d_2}^c})^*(\tla|Q|-2|Q\cap R_{d_2}|)
\\
\leq &
(|(f-m_f(Q)\chi_{Q\cap R_{d_2}^c}|+C\ep)^*(\tla|Q|-2|Q\cap R_{d_2}|)
\\
\leq &
((f-m_f(Q))\chi_{Q\cap R_{d_2}^c})^*(\tla|Q|-2|Q\cap R_{d_2}|)+C\ep
\\
\leq &
((f-m_f(Q))\chi_{Q})^*(\tla|Q|-2|Q\cap R_{d_2}|)+C\ep
\\
\leq &
((f-m_f(Q))\chi_{Q})^*(\la|Q|)+C\ep
={a}_{\lambda}(f;Q)+C\ep\ls\ep.
\end{split}
\ee
We have now completed the proof of (\ref{proof, 2}). It follows from the John-Str\"{o}mberg equivalence \eqref{e-J-S equiv}
that
\be
\|f-g_{\ep}\|_{\bmo}\ls\ep.
\ee
Combining this with (\ref{proof, 3}), we obtain
\be
\|f-g_{\ep,t}\|_{\bmo}\ls\ep,
\ee
where $g_{\ep,t}$ is a $C_c^{\infty}(\bbR)$ function, the implicit constant is independent of $\ep$. This completes the whole proof.
\end{proof}

\section{Necessity of compact commutators}\label{s4}
This section is devoted to the proof of Theorem \ref{theorem, necessity} and is divided into three subsections.
In Subsection \ref{s4.1}, for any given cube $Q$ and  real-valued measurable function $b$, we construct a function $f$ closely related to $Q$ and
obtain a lower bound of  the weighted $L^q$
norm of $(T_{\Omega,\,\al})_b^m(f)$ over certain subset of $Q$ in terms of $a_{\la}(b;Q)^m$. In Subsection \ref{s4.2},
for $b\in\bmo$ and any cube $Q$, we also obtain an upper bound of  the weighted $L^q$
norm of $(T_{\Omega,\,\al})_b^m(f)$ over the annulus $2^{d+1}Q\bs 2^dQ$ in terms of $2^{-\d dn/p}d^m$, where
$d\in\mathbb N$ large enough, $\delta$ is a positive constant depending on $w\in\apq$ and $f$ is aforementioned.
Using these upper and lower bounds, we further present the proof of Theorem \ref{theorem, necessity} in Subsection \ref{s4.3}.

\subsection{Lower estimates}\label{s4.1}
This part follows by the approach of Lerner-Ombrosi-Rivera-Rios \cite{LernerOmbrosiRivera17arxiv}.
For the self-containing of this paper, we reproof the following proposition for $\la\in (0,1)$ with a slight modification fitting our further proof.  Denote $Q_0:=[-1/2,1/2]^n$.

\begin{proposition}[\cite{LernerOmbrosiRivera17arxiv}]\label{proposition, key lower estimates}
  Let $\la\in (0,1)$, $b$ be a real-valued measurable function.
  Suppose that $\Om$ satisfy the assumption in Theorem \ref{theorem, necessity}.
  There exist $\ep_0>0$ and $k_0>10\sqrt{n}$ depending only on $\Om$ and $n$ such that the following holds.
  For every cube $Q$, there exists another cube $P$ with the same side length of $Q$ satisfying $|c_Q-c_P|=k_0l_Q$,
  and measurable sets $E\subset Q$ with $E=\frac{\la}{2}|Q|$, and $F\subset P$ with $|F|=\frac{1}{2}|Q|$,
  and $G\subset E\times F$ with $|G|\geq\frac{\la|Q|^2}{8}$ such that
  \bn
  \item ${a}_{\lambda}(b;Q)\leq |b(x)-b(y)|$ for all $(x,y)\in E\times F$;
  \item $\Om\left(\frac{x-y}{|x-y|}\right)$ and $b(x)-b(y)$ do not change sign in $E\times F$;
  \item $\left|\Om\left(\frac{x-y}{|x-y|}\right)\right|\geq \ep_0$ for all $(x,y)\in G$.
  \en
\end{proposition}
\begin{proof}
Without loss of generality, assume $\Omega$ is nonnegative on an open set of $\bbS^{n-1}$.
By the assumption of $\Om$, there exists a point $\th_0$ of approximate continuity of $\Omega$ such that
$\Omega(\th_0)=2\ep_0$ for some $\ep_0>0$ (see\cite[pp.46-47]{EvansGariepy92} for the definition of approximate continuity).
It follows from the definition of approximate continuity that for every
$\b\in (0,1)$, there exists a small constant $r_{\b}$ such that
\be
\s(\{\th\in B(\th_0,r_{\b})\cap \bbS^{n-1}: \Omega(\th)\geq \ep_0\})\geq (1-\b)\s(B(\th_0,r_{\b})\cap \bbS^{n-1})
\ee
Let $\G_{\b}$ be the cone containing all $x\in \bbR^n$ such that $x'\in B(\th_0,r_{\b})\cap \bbS^{n-1}$.

There exists a vector $v_{\b}=\frac{c_n\th_0}{r_{\b}}$, such that
\be
2Q_0+v_{\b}\in \G_{\b}.
\ee
For a cube fixed $Q$, we set
\be
P_{\b}:=Q-l_Q v_{\b},\  N:=\{x,y\in \bbR^n: \Om(x-y)<\ep_0\},\  N_x:=\{y\in \bbR^n: \Om(x-y)<\ep_0\}.
\ee
Here and in what follows, for any point $x_0\in\rn$ and sets $E, F\subset\rn$,
$E+x_0:=\{y+x_0: y\in E\}$ and $E-F:=\{x-y: x\in E, y\in F\}$.
Thus, $|c_Q-c_{P_{\b}}|=\frac{c_n}{r_{\b}}l_Q$.
Observe that $Q-P_{\b}\subset 2l_Q Q_0+l_Q v_{\b}$, we obtain
\be
\begin{split}
  |(Q\times P_{\b})\cap N|
  =
  \int_Q |P_{\beta}\cap N_x|dx
  = &
  \int_Q |(P_{\beta}-x)\cap N_0|dx
  \\
  \leq &
  \int_Q |(P_{\beta}-Q)\cap N_0|dx
  \\
  = &
  |Q|\cdot|(P_{\beta}-Q)\cap N_0|
  \\
  \leq &
  |Q|\cdot l_Q^n|(2Q_0+r_{\b})\cap N_0|
    \\
  \leq &
  |Q|\cdot l_Q^n\cdot c_n\cdot |v_{\b}|^{n-1}\cdot \b r_{\b}^{n-1}
  \leq c_n\b|Q|^2.
\end{split}
\ee
Take $\b=\b_0$ sufficiently small such that
\be
|(Q\times P_{\b_0})\cap N|\leq \frac{\la|Q|^2}{8},\ \    k_0:=\frac{c_n}{r_{\b_0}}>10\sqrt{n}.
\ee

  By the definition of $\tilde a_{\la}(f;Q)$, there exists a subset $\tQ$ of $Q$, such that $|\tQ|=\la|Q|$ and
  \be
  {a}_{\lambda}(b;Q)\leq |b(x)-m_b(P_{\b_0})|.
  \ee
  Then, by the definition of $m_b(P_{\b_0})$,
  there exists subsets $E\subset \tQ$ and $F\subset P_{\b_0}$ such that
  \be
  |E|=|\tQ|/2=\la|Q|/2,\ |F|=|P_{\b_0}|/2=|Q|/2,
  \ee
  and
  \be
  {a}_{\lambda}(b;Q)\leq |b(x)-b(y)|
  \ee
  for $x\in E, y\in F$, $b(x)-b(y)$ does not change sign in $E\times F$.
  Note $|E\times F|=\frac{\la |Q|^2}{4}$.
  The desired set can be chosen by
  \be
  G=E\times F\bs ((Q\times P_{\b_0})\cap N).
  \ee
  We have now completed this proof.
\end{proof}
\begin{proposition}\label{proposition, lower estimates}
 Let $\om\in \apq$, $\lambda\in(0, 1)$ and $b$ be a real-valued measurable function.
 For a given cube $Q$, let $P,E,F,G$ be the sets associated with $Q$ mentioned in Proposition \ref{proposition, key lower estimates}.
Then there exists a positive constant $C$ is independent of $Q$ such that for $f:=(\int_{F}\om(x)^pdx)^{-1/p}\chi_F$,
 and any measurable set $B$ with $|B|\leq \frac{\la}{8}|Q|$,
  \be
  \|(T_{\Omega,\,\al})_b^m(f)\|_{L^q(E\bs B,\om^q)}\ge C{a}_{\lambda}(b;Q)^m.
  \ee
\end{proposition}
\begin{proof}
  Using H\"{o}lder's inequality, we obtain
  \be
  \begin{split}
  \int_{E\bs B}|(T_{\Omega,\,\alpha})_b^m(f)(x)|dx
  = &
  \int_{E\bs B}|(T_{\Omega,\,\alpha})_b^m(f)(x)|\om(x)\cdot \om(x)^{-1}dx
  \\
  \leq &
  \left(\int_{E\bs B}|(T_{\Omega,\,\alpha})_b^m(f)(x)|^q\om^q(x)dx\right)^{1/q}\left(\int_{Q}\om^{-q'}(x)dx\right)^{1/q'}.
  \end{split}
  \ee
On the other hand, by Proposition \ref{proposition, key lower estimates},
  \be
  \begin{split}
  \int_{E\bs B}|(T_{\Omega,\,\alpha})_b^m(f)(x)|dx
  = &
  \left(\int_{F}\om(x)^pdx\right)^{-1/p}\cdot \int_{E\bs B}\left|\int_{F}(b(x)-b(y))^m\frac{\Om(x-y)}{|x-y|^{n-\al}}dy\right|dx
  \\
  = &
  \left(\int_{F}\om(x)^pdx\right)^{-1/p}\cdot \int_{E\bs B}\int_{F}|b(x)-b(y)|^m\frac{|\Om(x-y)|}{|x-y|^{n-\al}}dydx
  \\
  \geq &
  \left(\int_{F}\om(x)^pdx\right)^{-1/p}\cdot \int_{((E\bs B)\times F)\cap G}|b(x)-b(y)|^m\frac{|\Om(x-y)|}{|x-y|^{n-\al}}dxdy
  \\
  \gtrsim &
  \left(\int_{P}\om(x)^pdx\right)^{-1/p}\cdot |((E\bs B)\times F)\cap G|\cdot {a}_{\lambda}(b;Q)^m|Q|^{-1+n/\al}\ep_0.
  \end{split}
  \ee
Combining this and the fact
\be
|((E\bs B)\times F)\cap G|\geq |G|-|B||F|\geq \frac{\la}{8}|Q|^2-\frac{\la}{8}|Q|\cdot \frac{|Q|}{2}\geq \frac{\la}{16}|Q|^2,
\ee
we obtain
  \be
  \begin{split}
  \int_{E\bs B}|(T_{\Omega,\,\alpha})_b^m(f)(x)|dx
  \gtrsim
  \left(\int_{P}\om(x)^pdx\right)^{-1/p}\cdot {a}_{\lambda}(b;Q)^m|Q|^{1+n/\al}\ep_0.
  \end{split}
  \ee
This and the first estimate of this proof yields that
\begin{eqnarray}\label{proof, 4}
  &&\left(\int_{E\bs B}|(T_{\Omega,\,\alpha})_b^m(f)(x)|^q\om^q(x)dx\right)^{1/q}\nonumber\\
&& \quad \geq
  \left(\int_{Q}\om^{-q'}(x)dx\right)^{-1/q'}\cdot \left(\int_{P}\om^p(x)dx\right)^{-1/p}\cdot {a}_{\lambda}(b;Q)^m|Q|^{1+n/\al}\ep_0.
\end{eqnarray}

It follows from the fact $P\subset 4k_0Q$,
the definition of $\apq$ and the H\"{o}lder inequality that
\begin{eqnarray*}
 && \left(\frac{1}{|Q|}\int_{P}\om^p(x)dx\right)^{1/p}\left(\frac{1}{|Q|}\int_{Q}\om^{-q'}(x)dx\right)^{1/q'}\\
 && \quad\lesssim
  \left(\frac{1}{|Q|}\int_{Q}\om^p(x)dx\right)^{1/p}\left(\frac{1}{|Q|}\int_{Q}\om^{-q'}(x)dx\right)^{1/q'}\\
 && \quad\leq
  \left(\frac{1}{|Q|}\int_{Q}\om^q(x)dx\right)^{1/q}\left(\frac{1}{|Q|}\int_{Q}\om^{-p'}(x)dx\right)^{1/p'}\lesssim 1.
\end{eqnarray*}
This and (\ref{proof, 4}) yields that
\be
\left(\int_{E\bs B}|(T_{\Omega,\,\alpha})_b^m(f)(x)|^q\om^q(x)dx\right)^{1/q}
\gtrsim
{a}_{\lambda}(b;Q)^m,
\ee
which completes this proof.
\end{proof}
As a direct conclusion of Proposition \ref{proposition, lower estimates}, following is an extension of Theorem 1.1 (ii) in \cite{LernerOmbrosiRivera17arxiv}. Here, we only state the unweighted BMO case for our further proof, the weighted BMO case is still valid.
\begin{corollary}\label{corollary, necessity}
  Let $\om\in \apq$, $1<p,q<\infty$, $0\leq \al<n$, $1/q=1/p-\al/n$, $m\in \mathbb{Z}^+$.
  Let $\Om$ be a measurable function on $\bbS^{n-1}$, which does not change sign and is not equivalent to zero
  on some open subset of $\bbS^{n-1}$. If $(T_\al)_b^m$ is a bounded operator from $L^p(\om^p)$ to $L^q(\om^q)$, then
  $b\in \bmo$.
\end{corollary}

\subsection{Upper estimates}\label{s4.2}

\begin{proposition}\label{proposition, upper estimates}
Let $m\in \mathbb{Z}^+$, $b\in \bmo$, $\Om\in L^{\infty}(\bbS^{n-1})$, $\om\in \apq$.
For a cube $Q$, denote by $P,E,F,G$ the sets associated with $Q$ mentioned in Proposition \ref{proposition, key lower estimates}.
Let $f:=(\int_{F}\om(x)^pdx)^{-1/p}\chi_F$.
Then, there exists a positive constant $\d$ such that
\be
\|(T_{\Omega,\,\alpha})_b^m(f)\|_{L^q(2^{d+1}Q\bs 2^dQ, \om^q)}
\lesssim
2^{-\d dn/p}d^m.
\ee
for sufficient large $d$, where the implicit constant is independent of $d$ and $Q$.
\end{proposition}

\begin{proof}
Since $\om^p\in A_p$ and $F\subset P$ with $|F|=\frac{1}{2}|P|$, we have $\om^p(F)\gtrsim \om^p(P)$
\be
f(x)=\left(\int_{F}\om(x)^pdx\right)^{-1/p}\chi_F(x)\lesssim \left(\int_{P}\om(x)^pdx\right)^{-1/p}\chi_P(x)
\ee

A direct calculation yields that
\be
\begin{split}
  |(T_{\Omega,\,\alpha})_b^m(f)(x)|
  \lesssim &
  \left(\int_{P}\om(y)^pdy\right)^{-1/p}\int_{P}|b(x)-b(y)|^m\frac{|\Omega(x-y)|}{|x-y|^{n-\alpha}}dy
  \\
  = &
  \left(\int_{P}\om(y)^pdy\right)^{-1/p}\int_{P}|b(x)-b_P+b_P-b(y)|^m\frac{|\Omega(x-y)|}{|x-y|^{n-\alpha}}dy
  \\
  \leq &
  \left(\int_{P}\om(y)^pdy\right)^{-1/p}\sum_{i+j=m}C_m^i|b(x)-b_P|^i\int_{P}|b_P-b(y)|^j\frac{|\Omega(x-y)|}{|x-y|^{n-\alpha}}dy.
\end{split}
\ee
Recall the equivalent norm of $\bmo$
\be
\|b\|_{\bmo}\sim \sup\limits_{\tQ}\left(\int_{\tQ}|b(y)-b_{\tQ}|^jdy\right)^{1/j},\ \  (1\leq j<\infty).
\ee
For $x\in 2^{d+1}Q\bs 2^dQ$, observe that $|x-y|\sim 2^dl_Q$ for $y\in P$, we deduce that
\be
\begin{split}
  \int_{P}|b_P-b(y)|^j\frac{|\Omega(x-y)|}{|x-y|^{n-\alpha}}dy
\lesssim &
\frac{\|\Om\|_{L^{\infty}(\bbS^{n-1})}}{2^{d(n-\al)}|P|^{1-n/\al}}\int_{P}|b_P-b(y)|^jdy
\\
= &
\frac{\|\Om\|_{L^{\infty}(\bbS^{n-1})}}{2^{d(n-\al)}|P|^{-n/\al}}\frac{1}{|P|}\int_{P}|b_P-b(y)|^jdy
\\
\lesssim &
\frac{\|\Om\|_{L^{\infty}(\bbS^{n-1})}}{2^{d(n-\al)}|P|^{-n/\al}}\|b\|_{\bmo}^j.
\end{split}
\ee
Since $\om^q\in A^q$, there exists a small positive constant  $\ep\leq \ep_n/[\om^q]_{A_{\infty}}$,
such that
\be
\left(\frac{1}{|\tQ|}\int_{\tQ}\om^{q(1+\ep)}(x)dx\right)^{\frac{1}{1+\ep}}\leq \frac{2}{|\tQ|}\int_{\tQ}\om^q(x)dx\ \ \text{for all cubes}\  \tQ.
\ee
From this and the H\"{o}lder inequality, we obtain
\be
\begin{split}
  &\left\||b(\cdot)-b_P|^i\right\|_{L^q(2^{d+1}Q\bs 2^dQ, \om^q)}\\
  &\quad\leq
  \left\||b(\cdot)-b_P|^i\right\|_{L^q(2^{d+1}Q, \om^q)}
  \\
  &\quad\leq
  \left(\int_{2^{d+v}P}|b(x)-b_P|^{iq}\om^q(x)dx\right)^{1/q}
  \\
  &\quad\leq
  |2^{d+v}P|^{1/q}\left(\frac{1}{|2^{d+v}P|}\int_{2^{d+v}P}|b(x)-b_P|^{iq(1+\ep)'}dx\right)^{\frac{1}{q(1+\ep)'}}
  \left(\frac{1}{|2^{d+v}P|}\int_{2^{d+v}P}\om(x)^{q(1+\ep)}dx\right)^{\frac{1}{q(1+\ep)}}
  \\
 &\quad \lesssim
  |2^{d+v}P|^{1/q}\left(\frac{1}{|2^{d+v}P|}\int_{2^{d+v}P}|b(x)-b_P|^{iq(1+\ep)'}dx\right)^{\frac{1}{q(1+\ep)'}}
  \left(\frac{1}{|2^{d+v}P|}\int_{2^{d+v}P}\om(x)^{q}dx\right)^{\frac{1}{q}},
  \end{split}
\ee
where $v$ is a positive constant independent of $Q$ such that $2Q\subset 2^vP$.
By the fact
\be
|b_P-b_{2^{d+v}P}|\leq (d+v)\cdot 2^n\|b\|_{\bmo},
\ee
we have
\be
\begin{split}
  &\left(\frac{1}{|2^{d+v}P|}\int_{2^{d+v}P}|b(x)-b_P|^{iq(1+\ep)'}dx\right)^{\frac{1}{q(1+\ep)'}}
  \\
 & \quad\lesssim
  d^i+\left(\frac{1}{|2^{d+v}P|}\int_{2^{d+v}P}|b(x)-b_{2^{d+v}P}|^{iq(1+\ep)'}dx\right)^{\frac{1}{q(1+\ep)'}}
  \lesssim d^i+\|b\|_{\bmo}^i\lesssim d^i.
\end{split}
\ee
Combining this with the previous estimate yields
\be
\begin{split}
  &\|(T_{\Omega,\,\alpha})_b^m(f)\|_{L^q(2^{d+1}Q\bs 2^dQ, \om^q)}\\
& \quad \lesssim
  \frac{|2^{d+v}P|^{1/q}d^i}{2^{d(n-\al)}|P|^{-n/\al}}\left(\frac{1}{|2^{d+v}P|}\int_{2^{d+v}P}\om(x)^{q}dx\right)^{\frac{1}{q}}
  \left(\int_{P}\om(x)^pdx\right)^{-1/p}
  \\
 & \quad\lesssim
  2^{dn(1/q-1+\al/n)}d^i\left(\frac{1}{|2^{d+1}P|}\int_{2^{d+1}P}\om(x)^{q}dx\right)^{\frac{1}{q}}\left(\frac{1}{|P|}\int_{P}\om(x)^pdx\right)^{-1/p}.
\end{split}
\ee
By the property of $A_p$, there exists a small constant $\d>0$, such that $\om^p\in A_{p-\delta}$.
This and the doubling property of $A_{p-\delta}$ yields that
\be
\int_{2^{d+v}P}\om(x)^pdx\leq 2^{(d+v)n(p-\d)}[\om^p]_{A_{p-\d}}\int_{P}\om(x)^{p}dx,
\ee
which implies
\be
\left(\frac{1}{|P|}\int_{P}\om(x)^pdx\right)^{-1/p}
\lesssim
2^{-dn/p}2^{dn(1-\d/p)}\left(\frac{1}{|2^{d+1}P|}\int_{2^{d+1}P}\om(x)^{p}dx\right)^{-1/p}.
\ee
Thus,
\be
\begin{split}
  &\|(T_{\Omega,\,\alpha})_b^m(f)\|_{L^q(2^{d+1}Q\bs 2^dQ, \om^q)}
  \\
 & \quad\lesssim
  2^{dn(1/q-1+\al/n)}d^i2^{-dn/p}2^{dn(1-\d/p)}
  \left(\frac{1}{|2^{d+1}P|}\int_{2^{d+1}P}\om(x)^{q}dx\right)^{\frac{1}{q}}\left(\frac{1}{|2^{d+1}P|}\int_{2^{d+1}P}\om(x)^{p}dx\right)^{-1/p}
  \\
 & \quad\lesssim
  2^{-\d dn/p}d^i\left(\frac{1}{|2^{d+1}P|}\int_{2^{d+1}P}\om(x)^{q}dx\right)^{\frac{1}{q}}\left(\frac{1}{|2^{d+1}P|}\int_{2^{d+1}P}\om(x)^{p}dx\right)^{-1/p}
\end{split}
\ee
By the definition of $A_{p,q}$, we obtain
\be
\left(\frac{1}{|2^{d+1}P|}\int_{2^{d+1}P}\om(x)^qdx\right)^{1/q}\left(\frac{1}{|2^{d+1}P|}\int_{2^{d+1}P}\om(x)^{-p'}dx\right)^{1/p'}\lesssim 1.
\ee
This together with following ineqaulity
\be
1\lesssim \left(\frac{1}{|2^{d+1}P|}\int_{2^{d+1}P}\om(x)^{-p'}dx\right)^{1/p'}\left(\frac{1}{|2^{d+1}P|}\int_{2^{d+1}P}\om(x)^{p}dx\right)^{1/p}
\ee
yields that
\be
\left(\frac{1}{|2^{d+1}P|}\int_{2^{d+1}P}\om(x)^{q}dx\right)^{\frac{1}{q}}\left(\frac{1}{|2^{d+1}P|}\int_{2^{d+1}P}\om(x)^{p}dx\right)^{-1/p}\lesssim 1.
\ee
Thus, we get the desired estimate
\be
\|(T_{\Omega,\,\alpha})_b^m(f)\|_{L^q(2^{d+1}Q\bs 2^dQ, \om^q)}
\lesssim
2^{-\d dn/p}d^i
\lesssim
2^{-\d dn/p}d^m.
\ee
\end{proof}
\subsection{Proof of Theorem \ref{theorem, necessity}}\label{s4.3}

We only need to deal with the case that $b$ is real-valued.
If $(T_{\Omega,\,\alpha})_b^{m}$ is a compact operator from $L^p(\om^p)$ to $L^q(\om^q)$, then from Corollary \ref{corollary, necessity}, $b\in \bmo$.
To show $b\in\cmo$, we use a contradiction
argument via Theorem \ref{theorem, new characterization of CMO}. Observe that if $b\notin \cmo$,  $b$ does not satisfy at least one
of (1)-(3) in Theorem \ref{theorem, new characterization of CMO}. We further consider the following three cases.

First suppose that $b$ does not satisfies condition (1) in Theorem \ref{theorem, new characterization of CMO}.
There exist $\th_0>0$ and a sequence of cubes $\{Q_j\}_{j=1}^{\infty}$ with $|Q_j|\searrow 0$ as $j\rightarrow \infty$, such that
\be
{a}_{\lambda}(b,Q_j)\geq \th_0.
\ee
Given a cube $Q$, let $E$, $F$ be the cubes mentioned in Proposition \ref{proposition, key lower estimates}.
Let $f=(\int_{F}\om(x)^pdx)^{-1/p}\chi_F$.
Applying Propositions \ref{proposition, lower estimates} and \ref{proposition, upper estimates}, there exist positive constants $C_0$ and $d_0$
independent of $Q$,
such that
\be
\|(T_{\Omega,\,\alpha})_b^m(f)\|_{L^q(E\bs B,\om^q)}\geq 2C_0{a}_{\lambda}(b;Q)^m\ \text{for}\  |B|\leq \frac{\la}{8}|Q|,
\ee
and
\be
\|(T_{\Omega,\,\alpha})_b^m(f)\|_{L^q(\bbR^n\bs 2^{d_0}Q, \om^q)}
\leq
C_0\th_0^m.
\ee
Take a subsequence of $\{Q_j\}_{j=1}^{\infty}$, also denoted by $\{Q_{j}\}_{j=1}^{\infty}$ such that
\be
\frac{|Q_{j+1}|}{|Q_{j}|}\leq \min\{\frac{\la^2}{64}, 2^{-2d_0n}\}.
\ee
Denote
$
B_j: =\left(\frac{|Q_{j-1}|}{|Q_{j}|}\right)^{\frac{1}{2n}}Q_{j},\ \ j\geq 2.
$
It is easy to check
\be
\left(\frac{|Q_{j-1}|}{|Q_{j}|}\right)^{\frac{1}{2n}}\geq 2^{d_0},\ \ |B_{j}|\leq \frac{\la}{8}|Q_{j-1}|.
\ee
Moreover, for any $k>j$, we have
\be
2^{d_0}Q_k\subset B_k,\ \ |B_k|\leq \frac{\la}{8}|Q_{j}|.
\ee

Denote by $E_j$, $F_j$ the sets associated with $Q_j$ as mentioned in Proposition \ref{proposition, key lower estimates}.
Let
\be
f_j:=\left(\int_{F_j}\om(x)^pdx\right)^{-1/p}\chi_{F_j}.\
\ee
Again, for any $k>j\geq 1$, we obtain
\be
\|(T_{\Omega,\,\alpha})_b^m(f_j)\|_{L^q(E_j\bs B_k,\om^q)}\geq 2C_0{a}_{\lambda}(b;Q)^m\geq 2C_0\th_0^m
\ee
and
\be
\|(T_{\Omega,\,\alpha})_b^m(f_{k})\|_{L^q(E_j\bs B_k, \om^q)}
\leq
\|(T_{\Omega,\,\alpha})_b^m(f_{k})\|_{L^q(\bbR^n\bs 2^{d_0}Q_{k}, \om^q)}\leq C_0\th_0^m.
\ee
From this we get
\be
\begin{split}
 & \|(T_{\Omega,\,\alpha})_b^m(f_j)-(T_{\Omega,\,\alpha})_b^m(f_{k})\|_{L^q(\bbR^n, \om^q)}\\
& \quad \geq
  \|(T_{\Omega,\,\alpha})_b^m(f_j)-(T_{\Omega,\,\alpha})_b^m(f_{k})\|_{L^q(E_j\bs B_k, \om^q)}
  \\
&\quad  \geq
  \|(T_{\Omega,\,\alpha})_b^m(f_j)\|_{L^q(E_j\bs B_k, \om^q)}
  -\|(T_{\Omega,\,\alpha})_b^m(f_{k})\|_{L^q(E_j\bs B_k, \om^q)}\geq C_0\th_0^m,
\end{split}
\ee
which leads to a contradiction with the compactness of $(T_{\Omega,\,\alpha})_b^m$.

A similar contradiction argument is valid for the proof of condition (2), we omit the details here.
It remains to prove $b$ satisfies condition (3).

Assume that $b$ satisfies (2) but does not satisfy (3).
Hence, there exists $\th_1>0$ and a sequence of cube $\{\tQ_j\}_{j=1}^{\infty}$ with $|\tQ_j|\lesssim 1$  such that
\be
\tQ_j\cap R_j=\emptyset,\ \  {a}_{\lambda}(b,\tQ_j)\geq \th_1.
\ee

Denote by $\tE_j$, $\tF_j$ the sets associated with $\tQ_j$ as mentioned in Proposition \ref{proposition, key lower estimates}.
Let $C_0$ be the constant mentioned in the proof of condition (1).
Let
\be
\widetilde{f}_j:=\left(\int_{\tF_j}\om(x)^pdx\right)^{-1/p}\chi_{\tF_j}.\
\ee
By Proposition \ref{proposition, upper estimates}, there exists a positive constant $d_1$
independent of $\tQ_j$,
such that
\be
\|(T_{\Omega,\,\alpha})_b^m(\widetilde{f}_j)\|_{L^q(\bbR^n\bs 2^{d_1}\tQ_j, \om^q)}
\leq
C_0\th_1^m.
\ee
Take $d_2\geq d_1$ such that $\tE_j\subset 2^{d_2}\tQ_j$,
then take a subsequence of $\{\tQ_j\}_{j=1}^{\infty}$, still denoted by $\{\tQ_j\}_{j=1}^{\infty}$ such that
\be
2^{d_2}\tQ_i\cap 2^{d_2}\tQ_j=\emptyset,\ \ i\neq j.
\ee
For any $k\neq j$, note that $2^{d_1}\tQ_k\cap \tE_j\subset 2^{d_2}\tQ_k\cap 2^{d_2}\tQ_j=\emptyset$,
we apply Propositions \ref{proposition, lower estimates} and \ref{proposition, upper estimates} to get
\be
\|(T_{\Omega,\,\alpha})_b^m(\widetilde{f}_j)\|_{L^q(\tE_j\bs 2^{d_1}\tQ_k,\om^q)}
=\|(T_{\Omega,\,\alpha})_b^m(\widetilde{f}_j)\|_{L^q(\tE_j,\om^q)}
\geq 2C_0{a}_{\lambda}(b;\tQ_j)^m\geq 2C_0\th_1^m
\ee
and
\be
\|(T_{\Omega,\,\alpha})_b^m(\widetilde{f}_{k})\|_{L^q(\tE_j\bs 2^{d_1}\tQ_k, \om^q)}
\leq
\|(T_{\Omega,\,\alpha})_b^m(\widetilde{f}_{k})\|_{L^q(\bbR^n\bs 2^{d_1}\tQ_{k}, \om^q)}\leq C_0\th_1^m.
\ee
From this we get
\be
\begin{split}
 & \|(T_{\Omega,\,\alpha})_b^m(\widetilde{f}_j)-(T_{\Omega,\,\alpha})_b^m(\widetilde{f}_{k})\|_{L^q(\bbR^n, \om^q)}\\
 & \quad\geq
  \|(T_{\Omega,\,\alpha})_b^m(\widetilde{f}_j)-(T_{\Omega,\,\alpha})_b^m(\widetilde{f}_{k})\|_{L^q(\tE_j\bs 2^{d_1}\tQ_k, \om^q)}
  \\
 &\quad \geq
  \|(T_{\Omega,\,\alpha})_b^m(\widetilde{f}_j)\|_{L^q(\tE_j\bs 2^{d_1}\tQ_k, \om^q)}
  -\|(T_{\Omega,\,\alpha})_b^m(\widetilde{f}_{k})\|_{L^q(\tE_j\bs 2^{d_1}\tQ_k, \om^q)}\geq C_0\th_1^m,
\end{split}
\ee
which leads to a contradiction with the compactness of $(T_{\Omega,\,\alpha})_b^m$.

\section{Compactness of iterated commutators}\label{s5}

In this section, we first establish a result on the boundedness of multilinear commutator
 on weighted Lebesgue spaces in Subsection \ref{s5.1}. Using this result, we further present
 the proof of Theorem \ref{theorem, sufficiency} in Subsection \ref{s5.2}.

\subsection{boundedness of multilinear commutators}\label{s5.1}

In this subsection, we study the boundedness of multilinear commutators.
It is well known that the conjugation method is an efficient way to get the boundedness of commutators,
see \cite{PerezRivera16arxiv} for more details.
We will apply this method in our multi-linear case.
When $p\in (1,\infty)$ is fixed, we denote by $\s:= \om^{1-p'}$ the dual weight, which satisfies
$[\s]_{A_{p'}}^{1/p'}=[\om]_{A_p}^{1/p}$ by a simple calculation. Define
\be
(\om)_{A_{\infty}}:=\max\{[\om]_{A_{\infty}}, [\s]_{A_{\infty}}\}.
\ee

To begin with, we give the following two lemmas with slight modifications of Lemma 2.1 in \cite{Hytonen16ArchMath}.

\begin{lemma}\label{lemma, adjoint of Ap weight and BMO}
  Let $p\in (1,\infty)$, $\om\in A_p$, $b_j\in \bmo$ for $j=1,2,\cdots m$.
  There exists a constant $\k_{n,p,m}$ depending only on the indicated parameters,
  such that
  \be
  [e^{Re(\sum_{j=1}^m b_jz_j)}\om]_{A_p}\leq 4^p[\om]_{A_p}
  \ee
  for all $z_j\in \bbC$ with
  \be
  |z_j|\leq \frac{\k_{n,p,m}}{\|b_j\|_{\bmo}(1+(\om)_{A_{\infty}})}.
  \ee
\end{lemma}
\begin{proof}
From the reverse H\"{o}lder inequality and the John-Nirenberg inequality, there exists a constant $\eta_n$ such that
for any cube $Q$, $\eta\in(0, \leq \eta_n/(\om)_{A_{\infty}}]$ and $b\in \bmo$ with  $\|b\|_{\bmo}\leq \eta_n$,
\be
\left(\frac{1}{|Q|}\int_Q\om(x)^{1+\eta}dx\right)^{\frac{1}{1+\eta}}\leq 2\frac{1}{|Q|}\int_Q \om(x) dx,\ \ \
\left(\frac{1}{|Q|}\int_Q\s(x)^{1+\eta}dx\right)^{\frac{1}{1+\eta}}\leq 2\frac{1}{|Q|}\int_Q \s(x) dx
\ee
and
\be
\frac{1}{|Q|}\int_Q e^{|b(x)-b_Q|}dx\leq 2,
\ee
where $\s:= \om^{1-p'}$. Denote $r:=1+\eta_n/(\om)_{A_{\infty}}$ and $\mathbf{b}(x)\cdot\mathbf{z}:=\sum_{j=1}^mb_j(x)z_j$.
The H\"{o}lder inequality and reverse H\"{o}lder inequality imply that
\be
\begin{split}
  &\left(\frac{1}{|Q|}\int_Qe^{Re(\mathbf{b}(x)\cdot\mathbf{z})}\om(x) dx\right)\left(\frac{1}{|Q|}\int_Q(e^{Re(\mathbf{b}(x)\cdot\mathbf{z})}\om(x))^{1-p'} dx\right)^{p-1}
  \\
 & \quad\leq
  \left(\frac{1}{|Q|}\int_Q\om^{r}(x) dx\right)^{\frac{1}{r}}
  \left(\frac{1}{|Q|}\int_Q e^{r'Re(\mathbf{b}(x)\cdot\mathbf{z})} dx\right)^{\frac{1}{r'}}
  \left(\frac{1}{|Q|}\int_Q\s^{r}(x) dx\right)^{\frac{p-1}{r}}\\
  &\quad\quad\times
  \left(\frac{1}{|Q|}\int_Q e^{r'(1-p')Re(\mathbf{b}(x)\cdot\mathbf{z})} dx\right)^{\frac{p-1}{r'}}
  \\
&\quad  \leq
  2^p\left(\frac{1}{|Q|}\int_Q\om(x) dx\right)
  \left(\frac{1}{|Q|}\int_Q\s(x) dx\right)^{p-1}
  \left(\frac{1}{|Q|}\int_Q e^{r'Re(\mathbf{b}(x)\cdot\mathbf{z})} dx\right)^{\frac{1}{r'}}\\
  &\quad\quad\times
  \left(\frac{1}{|Q|}\int_Q e^{r'(1-p')Re(\mathbf{b}(x)\cdot\mathbf{z})} dx\right)^{\frac{p-1}{r'}}
  \\
 &\quad \leq
  2^p[\om]_{A_p}
  \left(\frac{1}{|Q|}\int_Q e^{r'Re(\mathbf{b}(x)\cdot\mathbf{z})} dx\right)^{\frac{1}{r'}}
  \left(\frac{1}{|Q|}\int_Q e^{r'(1-p')Re(\mathbf{b}(x)\cdot\mathbf{z})} dx\right)^{\frac{p-1}{r'}}
  \\
& \quad \leq
  2^p[\om]_{A_p}\prod_{j=1}^m\left(\frac{1}{|Q|}\int_Q e^{mr'Re(b_j(x)z_j)} dx\right)^{\frac{1}{mr'}}
  \prod_{j=1}^m\left(\frac{1}{|Q|}\int_Q e^{mr'(1-p')Re(b_j(x)z_j)} dx\right)^{\frac{p-1}{mr'}}.
\end{split}
\ee
Observing $\frac{1}{r'}= \frac{\eta_n}{\eta_n+(\om)_{A_{\infty}}}$,
we take $\k_{n,p,m}$ such that for any $j=1,2,\cdots, m$,
\be
\frac{\k_{n,p,m}}{\|b_j\|_{\bmo}(1+(\om)_{A_{\infty}})}\leq \min\left\{\frac{\eta_n}{mr'\|b_j\|_{\bmo}},\ \frac{\eta_n}{mr'(1-p')\|b_j\|_{\bmo}}\right\}.
\ee
Then for $|z_j|\leq \frac{\k_{n,p,m}}{\|b_j\|_{\bmo}(1+(\om)_{A_{\infty}})}$, we have
\begin{eqnarray*}
  &&\left(\frac{1}{|Q|}\int_Qe^{Re(\mathbf{b}(x)\cdot\mathbf{z})}\om(x) dx\right)\left(\frac{1}{|Q|}\int_Q(e^{Re(\mathbf{b}(x)\cdot\mathbf{z})}\om(x))^{1-p'} dx\right)^{p-1}\\
&& \quad \leq  2^p[\om]_{A_p}\prod_{j=1}^m\left(\frac{1}{|Q|}\int_Q e^{mr'[Re(b_j(x)z_j)-(Re(b_j(x)z_j))_Q]} dx\right)^{\frac{1}{mr'}}\\
&&\quad\quad \times \prod_{j=1}^m\left(\frac{1}{|Q|}\int_Q e^{mr'(1-p')[Re(b_j(x)z_j)-(Re(b_j(x)z_j))_Q]} dx\right)^{\frac{p-1}{mr'}}  \\
&& \quad \leq   2^p[\om]_{A_p}\prod_{j=1}^m 2^{\frac{1}{mr'}}
  \prod_{j=1}^m 2^{\frac{p-1}{mr'}}=2^{p+p/r'}[\om]_{A_p}\leq 4^p[\om]_{A_p}.
\end{eqnarray*}
\end{proof}

\begin{lemma}\label{lemma, adjoint of Apq weight and BMO}
  Let $p,q\in (1,\infty)$, $\om\in A_{p,q}$, $b_j\in \bmo$ for $j=1,2,\cdots m$. There exists a constant $\k_{n,p,q,m}$ depending only on the indicated parameters,
  such that
  \be
  [e^{Re(\sum_{j=1}^m b_jz_j)}\om]_{\apq}\leq 4^{1+q/p'}[\om]_{\apq}
  \ee
  for all $z_j\in \bbC$ with
  \be
  |z_j|\leq \frac{\k_{n,p,q,m}}{\|b_j\|_{\bmo}(1+(\om)_{A_{\infty}})}.
  \ee
\end{lemma}
\begin{proof}
Recall
  \be
\om\in \apq\Longleftrightarrow \om^q\in A_{1+\frac{q}{p'}},\ \ \ [\om]_{\apq}=[\om^q]_{A_{1+q/p'}}.
\ee
By Lemma \ref{lemma, adjoint of Ap weight and BMO}, there exists a constant $\k_{n,p,q,m}$ such that
\be
  [e^{qRe(\sum_{j=1}^m b_jz_j)}\om^q]_{A_{1+\frac{q}{p'}}}\leq 4^{1+q/p'}[\om^q]_{A_{1+\frac{q}{p'}}}
  \ee
  for all $z_j\in \bbC$ with
  \be
  |z_j|\leq \frac{\k_{n,p,q,m}}{\|b_j\|_{\bmo}(1+(\om^q)_{A_{\infty}})}.
  \ee
Hence,
\be
[e^{Re(\sum_{j=1}^m b_jz_j)}\om]_{\apq}
=[e^{qRe(\sum_{j=1}^m b_jz_j)}\om^q]_{A_{1+\frac{q}{p'}}}\leq 4^{1+q/p'}[\om^q]_{A_{1+\frac{q}{p'}}}=4^{1+q/p'}[\om]_{\apq}.\ee
\end{proof}

For a vector-valued function $\mathbf{b}:=(b_1,b_2,\cdots,b_m)$ with $b_j\in \bmo$, $1\leq j\leq m$,
and a $\r$-type operator $T_{K_{\al}}$,
the multi-linear commutator is defined by
\be
  (T_{K_{\al}})_{\mathbf{b}}^m:=[b_m,\cdots [b_2,[b_1, T_{K_\al}]]\cdots].
\ee

We then have the following conclusion.

\begin{theorem}\label{t-bdd of multilinear commutator}
Let $1<p,q<\infty$, $1/q=1/p-\al/n$ with $0\leq\al<n$, $\om\in A_{p,q}$.
For a  vector-valued function $\mathbf{b}:=(b_1,b_2,\cdots,b_m)$, $b_j\in \bmo$, $m\geq 2$, we have
  \be
  \|(T_{K_{\al}})_{\mathbf{b}}^m\|_{L^q(\om^q)}\lesssim \prod_{j=1}^m\|b_j\|_{\bmo}\|f\|_{L^p(\om^p)}.
  \ee
\end{theorem}

\begin{proof}
Denote $\mathbf{b}\cdot\mathbf{z}:=\sum_{j=1}^mb_jz_j$, where $\mathbf{z}:=(z_1,z_2,\cdots,z_m)\in \bbC^m$.
  Write $(T_{K_{\al}})_{\mathbf{b}}^m$ as a complex integral operator using the Cauchy integral theorem by
  \be
  \begin{split}
  (T_{K_{\al}})_{\mathbf{b}}^mf
  = &
  \left(\frac{\partial}{\partial z_m}
  \cdots \left(\frac{\partial}{\partial z_2}\left(\frac{\partial}{\partial z_1}e^{\mathbf{b}\cdot\mathbf{z}}T(fe^{-\mathbf{b}\cdot\mathbf{z}})|_{z_1=0}\right)\right)\bigg|_{z_2=0}\cdots\right)\bigg|_{z_m=0}
  \\
  =&
  \frac{1}{(2\pi i)^m}\int_{|z_m|=\ep_m}\cdots \int_{|z_1|=\ep_1}\frac{e^{\mathbf{b}\cdot\mathbf{z}}T(e^{-\mathbf{b}\cdot\mathbf{z}}f)}{\prod_{j=1}^m z_j^2}dz_1\cdots dz_m,
  \end{split}
  \ee
  where we take
  \be
  \ep_j:=\frac{\k_{n,p,q,m}}{\|b_j\|_{\bmo}(1+(\om)_{A_{\infty}})}
  \ee
  from Lemma \ref{lemma, adjoint of Apq weight and BMO}.
  This implies that for $\mathbf{b}\cdot\mathbf{z}$ such that $|z_j|=\ep_j$, $1\leq j\leq m$,
  \be
  e^{Re(\mathbf{b}\cdot\mathbf{z})}\om\in A_{p,q}.
  \ee
  Hence, by the well known boundedness of $T_{K_{\al}}$ (see \cite{Lacey17IsraelJMath, Lerner16NewYorkJMath}),
  \be
  \begin{split}
  \|(T_{K_{\al}})_{\mathbf{b}}^mf\|_{L^q(\om^q)}
 & \leq
  \frac{1}{\prod_{j=1}^m\ep_j}\sup_{|z_j|=\ep_j}\|e^{\mathbf{b}\cdot\mathbf{z}}T(e^{-\mathbf{b}\cdot\mathbf{z}}f)\|_{L^q(\om^q)}
  \\
& \leq
  \frac{1}{\prod_{j=1}^m\ep_j}\sup_{|z_j|=\ep_j}\|e^{Re(\mathbf{b}\cdot\mathbf{z})}T(e^{-\mathbf{b}\cdot\mathbf{z}}f)\|_{L^q(\om^q)}
  \\
& =
  \frac{1}{\prod_{j=1}^m\ep_j}\sup_{|z_j|=\ep_j}\|T(e^{-\mathbf{b}\cdot\mathbf{z}}f)\|_{L^q((e^{Re(\mathbf{b}\cdot\mathbf{z})}\om)^q)}
  \\
 & \lesssim
  \frac{1}{\prod_{j=1}^m\ep_j}\sup_{|z_j|=\ep_j}\|e^{-\mathbf{b}\cdot\mathbf{z}}f\|_{L^p((e^{Re(\mathbf{b}\cdot\mathbf{z})}\om)^p)}
   =
  \frac{1}{\prod_{j=1}^m\ep_j}\|f\|_{L^p(\om^p)}.
  \end{split}
  \ee
By the definition of $\ep_j$, we get the desired conclusion.
\end{proof}

\subsection{Proof of Theorem \ref{theorem, sufficiency}}\label{s5.2}

In this subsection, we establish the compactness of iterated commutators. Firstly, we recall the following weighted
Fr\'echet-Kolmogorov theorem obtained in \cite{ClopCruz13AASFM}.
\begin{lemma}\label{lemma, criterion of precompact}
  Let $p\in (1,\infty)$, $\om\in A_p$. A subset $E$ of $L^p(\om)$ is precompact (or totally bounded)
  if the following statements hold:
  \bn[(a)]
  \item $E$ is uniformly bounded, i.e., $sup_{f\in E}\|f\|_{L^p(\om)}\lesssim 1$;
  \item $E$ uniformly vanishes at infinity, that is,
  \be
  \lim_{N\rightarrow \infty}\int_{|x|>N}|f(x)|^p\om(x)dx\rightarrow 0,\ \text{uniformly for all}\ f\in E.
  \ee
  \item $E$ is uniformly equicontinuous, that is,
  \be
  \lim_{\r\rightarrow 0}\sup_{y\in B(0,\r)}\int_{\bbR^n}|f(x+y)-f(x)|^p\om(x)dx\rightarrow 0,\ \text{uniformly for all}\ f\in E.
  \ee
  \en
\end{lemma}

\begin{proof}[Proof of Theorem \ref{theorem, sufficiency}]
By the definition of compact operator, we will verify the set
$$A(K,b):=\{(T_{K_\al})_b^m(f): \|f\|_{L^p(\om^p)}\leq 1\}$$
 is precompact.
Suppose $b\in \cmo$. For any $\ep>0$ there exists a $C_c^{\infty}$ function $b_{\ep}$ such that
\be
\|b-b_{\ep}\|_{\bmo}<\ep.
\ee
Recall an elementary formula
\be
b^m-b_{\ep}^m=(b-b_{\ep})(b^{m-1}+b^{m-2}b_{\ep}+\cdots+b_{\ep}^{m-1}).
\ee
From this and Theorem \ref{t-bdd of multilinear commutator}, we obtain
\be
\|(T_{K_\al})_b^m-(T_{K_\al})_{b_{\ep}}^m\|_{L^p(\om^p)\rightarrow L^q(\om^q)}\lesssim \ep\|b\|_{\bmo}^{m-1}.
\ee
Thus, in order to verify the set $A(K,b)$ is precompact, or equivalently, totally bounded on $L^q(\om^q)$,
we only need to consider the case of $b\in C_c^\fz(\rn)$.

Moreover, observe that for a general weight $\om$, the norm of the space $L^q(\om^q)$ is not invariant under translation.
Keep the assumption $b\in C_c^\fz(\rn)$, a further reduction for the kernel $K$ is needed.
As mentioned in \cite{ClopCruz13AASFM}, the idea of considering truncated operators to prove compactness results goes back to Krantz and Li \cite{KrantzLi01JMAAb}.
Take $\va\in C_c^{\infty}(\bbR^n)$ supported on $B(0,1)$ such that $\va=1$ on $B(0,1/2)$, $0\leq \va\leq 1$.
Let $\va_{\d}(x):=\va(\frac{x}{\d})$, $K_{\al}^{\d}(x,y): =K_{\al}(x,y)\cdot (1-\va_{\d}(x-y))$.
For $|x-x'|<|x-y|/2$, we have
\ben\label{proof, 5}
\begin{split}
  &|K_{\al}^{\d}(x,y)-K_{\al}^{\d}(x',y)|+|K_{\al}^{\d}(y,x)-K_{\al}^{\d}(y,x')|
  \\
 &\quad \leq
  |K_{\al}(x,y)-K_{\al}(x',y)|+|K_{\al}(y,x)-K_{\al}(y,x')|
  \\
& \quad\quad +
  (|K_{\al}(x,y)|+|K_{\al}(y,x)|)\cdot |\va_{\d}(x-y)-\va_{\d}(x'-y)|
  \\
  &\quad\ls
\frac{1}{|x-y|^{n-\al}}\r\left(\frac{|x-x'|}{|x-y|}\right)+\frac{1}{|x-y|^{n-\al}}|\va_{\d}(x-y)-\va_{\d}(x'-y)|.
\end{split}
\een
It follows from the mean value formula that for some $\th\in (0,1)$,
\be
\begin{split}
  |\va_{\d}(x-y)-\va_{\d}(x'-y)|
  \leq
  \left|\nabla \va\left(\frac{(1-\th)x+\th x'-y}{\d}\right)\right|\cdot \frac{|x-x'|}{\d}.
\end{split}
\ee
From this and
\be
|\nabla \va(x)|\lesssim \chi_{1/2\leq|x|\leq 1}(x),\ \ |(1-\th)x+\th x'-y|\sim |x-y|,
\ee
we obtain
\be
|\va_{\d}(x-y)-\va_{\d}(x'-y)|\lesssim \frac{|x-x'|}{|x-y|}.
\ee
Combining this estimate with (\ref{proof, 5}) yields
\be
\begin{split}
  |K_{\al}^{\d}(x,y)-K_{\al}^{\d}(x',y)|+|K_{\al}^{\d}(y,x)-K_{\al}^{\d}(y,x')|
\lesssim &
\frac{1}{|x-y|^{n-\al}}\left(\r\left(\frac{|x-x'|}{|x-y|}\right)+\frac{|x-x'|}{|x-y|}\right)
\\
\lesssim &
\frac{1}{|x-y|^{n-\al}}\tilde\r\left(\frac{|x-x'|}{|x-y|}\right),
\end{split}
\ee
where the function $\tilde\r(t):= \r(t)+t$ also satisfies the Dini condition.
Moreover,
\be
\begin{split}
  |(T_{K_\al^{\d}})_b^mf(x)-(T_{K_\al})_b^mf(x)|
  \leq &
  \left|\int_{\bbR^n}(b(x)-b(y))^m\va_{\d}(x-y)K_{\al}(x,y)f(y)dy\right|
  \\
  \leq &
  \int_{|x-y|\leq \d}\frac{|f(y)|}{|x-y|^{n-\al-m}}dy.
\end{split}
\ee
By the usual dyadic decomposition method, we get
\be
\begin{split}
  \int_{|x-y|\leq \d}\frac{|f(y)|}{|x-y|^{n-\al-m}}dy
= &
\sum_{j=0}^{\infty}\int_{2^{-(j+1)}\d \leq|x-y|\leq 2^{-j}\d}\frac{|f(y)|}{|x-y|^{n-\al-m}}dy
\\
\leq &
\sum_{j=0}^{\infty}(2^{-j}\d)^m\int_{2^{-(j+1)}\d \leq|x-y|\leq 2^{-j}\d}\frac{|f(y)|}{|x-y|^{n-\al}}dy
\\
\lesssim &
\sum_{j=0}^{\infty}2^{-jm}\d^m M_{\al}(f)(x)\lesssim \d^mM_{\al}(f)(x).
\end{split}
\ee
From the above two estimates we get
\be
\|(T_{K_\al^{\d}})_b^mf-(T_{K_\al})_b^mf\|_{L^q(\om^q)}\lesssim \d^m\|M_{\al}f\|_{L^q(\om^q)}\lesssim \d^m\|f\|_{L^p(\om^p)}.
\ee
Since $\d$ can be chosen arbitrarily small, we only need to verify $A(K^{\d},b)$ is totally bounded,
where $\d>0$, $b\in C_c^{\infty}(\bbR^n)$. Thanks to Lemma \ref{lemma, criterion of precompact}, we only need to check the conditions (a)-(c)
for $A(K^{\d},b)$.

Without loss of generality, we assume that $b$ is supported in a cube $Q$ centered at the origin.
By the boundedness of $(T_{K_{\al}^{\d}})_b^m$, $A(K^{\d},b)$ is a bounded set in $L^q(\om^q)$, which shows the correction of condition (a).

For $x\in (2Q)^c$,
\be
\begin{split}
  |(T_{K_{\al}^{\d}})_b^m(f)(x)|
  = &
  \left|\int_{\bbR^n}(b(y))^mK_{\al}^{\d}(x,y)f(y)dy\right|
  \\
  \lesssim &
  \frac{\|b\|_{L^{\infty}}^m}{|x|^{n-\al}}\int_{Q}|f(y)|dy
  \\
  \lesssim &
  \frac{\|b\|_{L^{\infty}}^m}{|x|^{n-\al}}\|f\|_{L^p(\om^p)}\left(\int_{Q}\om^{-p'}(x)dx\right)^{1/p'}.
\end{split}
\ee
Take $N>2$,
\be
\begin{split}
  \left(\int_{(2^NQ)^c}|(T_{K_{\al}^{\d}})_b^m(f)(x)|^q\om(x)^qdx\right)^{1/q}
  \lesssim
  \left(\int_{(2^NQ)^c}\frac{\om(x)^q}{|x|^{q(n-\al)}}dx\right)^{1/q}\left(\int_{Q}\om^{-p'}(x)dx\right)^{1/p'}.
\end{split}
\ee
Since $\om^q\in A_{\frac{q(n-\al)}{n}}$,
there exists a positive constant $\d>0$ such that $\om^q\in A_{{\frac{q(n-\al)}{n}}-\d}$.
By the doubling property of $A_{{\frac{q(n-\al)}{n}}-\d}$ we obtain
\be
\int_{2^{d}Q}\om(x)^qdx\leq 2^{dq(n-\al)-dn\d}[\om^q]_{A_{\frac{q(n-\al)}{n}-\d}}\int_{Q}\om(x)^{q}dx,
\ee
which implies
\be
\int_{2^{d+1}Q\bs 2^dQ}\frac{\om(x)^q}{|x|^{q(n-\al)}}dx
\lesssim
\frac{2^{dq(n-\al)-dn\d}}{2^{dq(n-\al)}}=2^{-dn\d}.
\ee
Thus,
\be
\begin{split}
  \left(\int_{(2^NQ)^c}|(T_{K_{\al}^{\d}})_b^m(f)(x)|^q\om(x)^qdx\right)^{1/q}
  \lesssim &
  \left(\sum_{j=0}^{\infty}\int_{2^{N+j+1}Q\bs 2^{N+j}Q}\frac{\om(x)^q}{|x|^{q(n-\al)}}dx\right)^{1/q}
  \\
  \lesssim &
  \left(\sum_{j=0}^{\infty}2^{-(N+j)n\d}\right)^{1/q}
  =
  2^{-Nn\delta}\left(\sum_{j=0}^{\infty}2^{-jn\d}\right)^{1/q}
\end{split}
\ee
which tends to zero as $N$ tends to infinity. This proves condition (b).

It remains to prove that $A(K^{\d},b)$ is equicontinuous in $L^q(\om^q)$. Assume that $\|f\|_{\lpwp}=1$. Take $z\in \bbR^n$ with $|z|\leq \frac{\d}{8}$, then
\be
\begin{split}
  &(T_{K_\al^{\d}})_b^m(f)(x+z)-(T_{K_\al^{\d}})_b^m(f)(x)
  \\
  &\quad=
  \int_{\bbR^n}(b(x+z)-b(y))^m(K_{\al}^{\d}(x+z,y)-K_{\al}^{\d}(x,y))f(y)dy
  \\
  &\quad\quad +
  \int_{\bbR^n}\big((b(x+z)-b(y))^m-(b(x)-b(y))^m\big)K_{\al}^{\d}(x,y)f(y)dy
  \\
&\quad=: I_1(x,z)+I_2(x,z).
\end{split}
\ee
We start the estimate of the first term.
Observing that $K_{\al}^{\d}(x+z,y)$ and $K_{\al}^{\d}(x,y)$ both vanish when $|x-y|\leq \frac{\d}{4}$,
then
\be
\begin{split}
  |I_1(x,z)|
  \leq &
  \int_{|x-y|\geq \d/4}|b(x+z)-b(y)|^m|K_{\al}^{\d}(x+z,y)-K_{\al}^{\d}(x,y)|\cdot|f(y)|dy
  \\
  \lesssim &
  \int_{|x-y|\geq \d/4}\frac{1}{|x-y|^{n-\al}}\tilde\r\left(\frac{|z|}{|x-y|}\right)|f(y)|dy
  \\
  = &
  \sum_{j=0}^{\infty}\int_{2^{j-2}\d\leq |x-y|\leq 2^{j-1}\d}\frac{1}{|x-y|^{n-\al}}\tilde\r\left(\frac{|z|}{|x-y|}\right)|f(y)|dy
  \\
  \leq &
  \sum_{j=0}^{\infty}\tilde\r\left(\frac{2^{2-j}|z|}{\d}\right)
  \int_{2^{j-2}\d\leq |x-y|\leq 2^{j-1}\d}\frac{1}{|x-y|^{n-\al}}|f(y)|dy
  \lesssim
  \sum_{j=0}^{\infty}\tilde\r\left(\frac{2^{2-j}|z|}{\d}\right)M_{\al}(f)(x).
\end{split}
\ee
From this and the following estimate
\be
\begin{split}
  \sum_{j=0}^{\infty}\tilde\r\left(\frac{2^{2-j}|z|}{\d}\right)
  \leq &
  \sum_{j=0}^{\infty}\int_{2^{-j}}^{2^{-j+1}}\frac{\tilde\r(4t|z|/\d)}{2^{-j}}dt
  \lesssim
  \int_{0}^2\frac{\tilde\r(4t|z|/\d)}{t}dt
  =
  \int_{0}^{8|z|/\d}\frac{\tilde\r(t)}{t}dt,
\end{split}
\ee
we obtain
\be
\|I_1(\cdot,z)\|_{L^q(\om^q)}\lesssim \left(\int_{0}^{8|z|/\d}\frac{\tilde\r(t)}{t}dt\right)\|M_{\al}f\|_{L^q(\om^q)}
\lesssim \left(\int_{0}^{8|z|/\d}\frac{\tilde\r(t)}{t}dt\right)\|f\|_{L^p(\om^p)}\leq \int_{0}^{8|z|/\d}\frac{\tilde\r(t)}{t}dt.
\ee
This shows that $\|I_1(\cdot,z)\|_{L^q(\om^q)}\rightarrow 0$ as $|z|\rightarrow 0$.

Divide the second term $I_{2}(x,z)$ by
\be
\begin{split}
  I_2(x,z)
  = &
  \int_{\bbR^n}\big((b(x+z)-b(y))^m-(b(x)-b(y))^m\big)K_{\al}^{\d}(x,y)f(y)dy
  \\
  = &
  \int_{|x-y|>\d}\big((b(x+z)-b(y))^m-(b(x)-b(y))^m\big)K_{\al}(x,y)f(y)dy
  \\
  & +
  \int_{\d/2\leq |x-y|\leq\d}\big((b(x+z)-b(y))^m-(b(x)-b(y))^m\big)K^{\d}_{\al}(x,y)f(y)dy
  \\
  =: &
  I_{2,1}(x,z)+I_{2,2}(x,z).
\end{split}
\ee
Next, we write
\be
\begin{split}
  \left(b(x+z)-b(y))^m-(b(x)-b(y)\right)^m
  = &
  \left(b(x+z)-b(x)+b(x)-b(y))^m-(b(x)-b(y)\right)^m
  \\
  = &
  \sum_{i=1}^mC_m^i(b(x+z)-b(x))^i(b(x)-b(y))^{m-i}
  \\
  = &
  \sum_{i=1}^mC_m^i(b(x+z)-b(x))^i\sum_{j=0}^{m-i}C_{m-i}^jb(x)^jb(y)^{m-i-j}
\end{split}
\ee
Hence,
\be
\begin{split}
  |I_{2,1}(x,z)|
  \leq &
  \sum_{i=1}^mC_m^i|b(x+z)-b(x)|^i\sum_{j=0}^{m-i}C_{m-i}^j|b(x)|^j\left|\int_{|x-y|>\d}K_{\al}(x,y)b(y)^{m-i-j}f(y)dy\right|
  \\
  \leq &
  \sum_{i=1}^mC_m^i|b(x+z)-b(x)|^i\sum_{j=0}^{m-i}C_{m-i}^j|b(x)|^j|T_{K_{\al}}^*(b^{m-i-j}f)(x)|
  \\
  \lesssim &
  |z|\sum_{j=0}^{m-i}C_{m-i}^j|T_{K_{\al}}^*(b^{m-i-j}f)(x)|.
\end{split}
\ee
From this and the $L^p(\om^p)\rightarrow L^q(\om^q)$ boundedness of $T_{K_{\al}}^*$, we obtain
\be
\begin{split}
  \|I_{2,1}(\cdot,z)\|_{L^q(\om^q)}
  \lesssim
  |z|\cdot \sum_{j=0}^{m-i}C_{m-i}^j\|T_{K_{\al}}^*(b^{m-i-j}f)\|_{L^q(\om^q)}
  \lesssim
  |z|\cdot \|f\|_{L^p(\om^p)}
  \leq |z|.
\end{split}
\ee
On the other hand,
\be
\begin{split}
  \left|\int_{\d/2\leq |x-y|\leq\d}K^{\d}_{\al}(x,y)b(y)^{m-i-j}f(y)dy\right|
  \lesssim &
  \int_{\d/2\leq |x-y|\leq\d}|K^{\d}_{\al}(x,y)|\cdot |f(y)|dy
  \\
  \lesssim &
  \frac{1}{\d^{n-\al}}\int_{\d/2\leq |x-y|\leq\d}|f(y)|dy\lesssim M_{\al}(f)(x).
\end{split}
\ee
This and the similar estimate of $I_{2,1}$ imply that
\be
\begin{split}
  |I_{2,2}(x,z)|
  \leq &
  \sum_{i=1}^mC_m^i|b(x+z)-b(x)|^i\sum_{j=0}^{m-i}C_{m-i}^j|b(x)|^j\left|\int_{\d/2\leq |x-y|\leq\d}K^{\d}_{\al}(x,y)b(y)^{m-i-j}f(y)dy\right|
  \\
  \lesssim &
  |z|M_{\al}(f)(x).
\end{split}
\ee
Hence,
\be
\|I_{2,2}(\cdot,z)\|_{L^q(\om^q)}\lesssim |z|\|M_{\al}f\|_{L^q(\om^q)}\lesssim |z|\|f\|_{L^p(\om^p)}\leq |z|.
\ee
It follows from above estimates of $I_1$, $I_{2,1}$ and $I_{22}$ that
\be
  \|(T_{K_\al^{\d}})_b^m(f)(\cdot+z)-(T_{K_\al^{\d}})_b^m(f)(\cdot)\|_{L^q(\om^q)}\rightarrow 0,
\ee
as $|z|\rightarrow 0$, uniformly for all $f$ with $\|f\|_{L^p(\om^p)}\leq 1$.
\end{proof}

\begin{remark}
  By a slight modification of the above proof, one can verify that the result of Theorem \ref{theorem, sufficiency} is still valid in the multilinear setting. More precisely, for
  a sequence of functions $b_j\in \cmo$, $j=1,2,\cdots,m$, $(T_{K_{\al}})_{\mathbf{b}}^m$
  is a compact operator from $L^p(\om^p)$ to $L^q(\om^q)$.
\end{remark}

\end{document}